\documentclass[leqno,12pt]{amsart} 
\usepackage{amsmath, amssymb}
\usepackage{amsthm} 
\usepackage[all]{xy}
\usepackage{enumerate}
\usepackage{color}	
\renewcommand{\thefootnote}{} 
\theoremstyle{plain} 
\newtheorem{theorem}{\indent\sc Theorem}[section]
\newtheorem{lemma}[theorem]{\indent\sc Lemma}
\newtheorem{corollary}[theorem]{\indent\sc Corollary}
\newtheorem{proposition}[theorem]{\indent\sc Proposition}
\newtheorem{claim}[theorem]{\indent\sc Claim}
\theoremstyle{definition} 
\newtheorem{definition}[theorem]{\indent\sc Definition}
\newtheorem{remark}[theorem]{\indent\sc Remark}
\newtheorem{example}[theorem]{\indent\sc Example}

\newcommand{\C}{\mathbb{C}}
\newcommand{\R}{\mathbb{R}}
\newcommand{\Z}{\mathbb{Z}}
\newcommand{\N}{\mathbb{N}}
\newcommand{\K}{\mathfrak{K}}

\newcommand{\Pen}{\mathrm{Pen}}

\def\Im{\mathop{\mathrm{Im}}\nolimits}

\newcommand{\abs}[1]{\lvert#1\rvert}
\newcommand{\norm}[1]{\lVert#1\rVert}

\newcommand{\limone}{\varprojlim{\!}^1}
\newcommand{\U}{\mathcal{U}}
\newcommand{\V}{\mathcal{V}}
\newcommand{\Horo}{\mathcal{H}}

\newcommand{\EG}{\underline{E}G}
\newcommand{\EP}{\underline{E}P}
\newcommand{\Xaug}{X(G,\mathbb{P},\mathcal{S})}
\newcommand{\barXaug}{\overline{X}(G,\mathbb{P},\mathcal{S})}
\newcommand{\EX}{EX(G,\mathbb{P})}
\newcommand{\famP}{\mathbb{P}}

\newcommand{\grad}{\mathbf{d}}

\newcommand{\rsHig}{\mathfrak{c}^r}
\newcommand{\barHig}{\bar{\mathfrak{c}}^r}

\newcommand{\dX}{\partial X}

\newcommand{\mX}{Y}
\makeatletter
\def\address#1#2{\begingroup
\noindent\parbox[t]{7.8cm}{%
\small{\scshape\ignorespaces#1}\par\vskip1ex
\noindent\small{\itshape E-mail address}%
\/: #2\par\vskip4ex}\hfill%
\endgroup}%
\makeatother
\title{\uppercase{Coronae of relatively hyperbolic groups and coarse
cohomologies}}
\author{
\textsc{Tomohiro Fukaya, Shin-ichi Oguni} 
}
\date{} 
\begin{document}

\maketitle

\footnote{ 
2010 \textit{Mathematics Subject Classification}.
Primary 58J22; Secondary 20F67, 20F65.
}
\footnote{ 
\textit{Key words and phrases}. 
coarse cohomology, coarse assembly map, coarse co-assembly map,
relatively hyperbolic group, corona
}
\footnote{ 
T.Fukaya and S.Oguni were supported by Grant-in-Aid for Young Scientists (B)
(23740049) and (24740045), respectively, 
from Japan Society of Promotion of Science
}
\renewcommand{\thefootnote}{\fnsymbol{footnote}} 


\begin{abstract}
We construct a corona of a relatively hyperbolic group by blowing-up all
 parabolic points of its Bowditch boundary. We relate the $K$-homology
 of the corona with the $K$-theory of the Roe algebra, via the coarse
 assembly map. We also establish a dual theory, that is, we relate the
 $K$-theory of the corona with the $K$-theory of the reduced stable
 Higson corona via the coarse co-assembly map.  For that purpose, we
 formulate generalized coarse cohomology theories.  As an application,
 we give an explicit computation of the $K$-theory of the Roe-algebra
 and that of the reduced stable Higson corona of the fundamental groups of
 closed 3-dimensional manifolds and of pinched negatively curved complete Riemannian 
 manifolds with finite volume.
\end{abstract}

\section{Introduction} 





\subsection{The coarse assembly map and its dual}
The coarse category is a category whose objects are proper metric spaces
and whose morphisms are close classes of coarse maps.  Let $X$ be a
proper metric space. There are two covariant functors $X \mapsto
KX_*(X)$ and $X \mapsto K_*(C^*(X))$ from the coarse category to the
category of $\Z_2$-graded Abelian groups. Here the $\Z_2$-graded Abelian
group $KX_*(X)$ is called the {\itshape coarse $K$-homology} of $X$, and the
$C^*$-algebra $C^*(X)$ is called the {\itshape Roe algebra} of $X$.  
Roe \cite{MR1147350} constructed the following {\itshape coarse assembly map}
\begin{align*}
 \mu_*&\colon KX_*(X) \to K_*(C^*(X)),
\end{align*}
which is a natural transformation from the coarse
$K$-homology to the $K$-theory of the Roe algebra. 
For detail, see also \cite{MR1388312}, \cite{MR1344138} and
\cite{MR1817560}.


On the other hand, there are two contravariant functors
$X\mapsto KX^*(X)$ and $X\mapsto K_*(\rsHig(X))$. Here the $\Z_2$-graded
Abelian group $KX^*(X)$ is called the {\itshape coarse $K$-theory} of $X$ and 
the $C^*$-algebra $\rsHig(X)$ is called the {\itshape reduced stable
Higson corona} of $X$.
Emerson and Meyer \cite{MR2225040} constructed a dual of the coarse
assembly map, which is called the {\itshape coarse co-assembly map},
\begin{align*}
 \mu^*&\colon K_{*+1}(\rsHig(X)) \to  KX^*(X).
\end{align*}
In fact, $\mu^*$ is a natural transformation from the $K$-theory
of the reduced stable Higson corona to the coarse $K$-theory with 
the grading shifted by one. Those assembly maps are closely related to the
analytic Novikov conjecture. See \cite[Section 12.6]{MR1817560} and
\cite{EM-descent-Principle} for details.


In this paper, we study the case of
relatively hyperbolic groups with word metrics.
\begin{theorem}
\label{th:coarse_co-assembly-map} 
Let $G$ be a finitely generated group which is
 hyperbolic relative to a finite family of infinite subgroups
 $\famP=\{P_1,\dots,P_k\}$.
 Suppose that each subgroup $P_i$ admits a finite $P_i$-simplicial
 complex which is a universal space for proper actions. Then
\begin{enumerate}[\indent$($a$)$]
 \item \label{item:main-th-hm}
if for all $i = 1,\dots, k$, the coarse assembly maps 
$\mu_* \colon KX_*(P_i) \to K_*(C^*(P_i))$ are isomorphisms,
then so is the coarse assembly map 
$\mu_* \colon KX_*(G) \to K_*(C^*(G))$,
 \item \label{item:main-th-coh}
if for all $i = 1,\dots, k$, the coarse co-assembly maps 
$\mu^*\colon K_{*+1}(\rsHig(P_i)) \to KX^*(P_i)$ 
are isomorphisms, then so is the coarse co-assembly map
$\mu^*\colon K_{*+1}(\rsHig(G)) \to KX^*(G)$.
\end{enumerate}
\end{theorem}
The authors proved the statement $($a$)$ in
\cite{relhypgrp}. In this paper, we prove the statement $($b$)$.

\subsection{Coarse compactification}
Let $X$ be a non-compact proper metric
space. The {\itshape Higson compactification} 
$hX$ of $X$ is the maximal ideal space of the $C^*$-algebra of
$\C$-valued, continuous, bounded functions on $X$ of vanishing
variation. (See Definition~\ref{def:Higson_function}.) 
The {\itshape Higson corona}
of $X$ is $\nu X = hX \setminus X$. 
A corona of $X$ is a pair $(W,\zeta)$ of a compact
metrizable space $W$ and a continuous map $\zeta \colon \nu X \to W$.
When $\zeta$ is surjective, we obtain a compactification $X\cup W$.
(See Section~\ref{sec:higs-comp}.) 


Let $(W,\zeta)$ be a corona of $X$. Then there are
certain transgression maps
\begin{align}
\label{Tr:1} &{T_W} \colon KX_*(X) \to \tilde{K}_{*-1}(W);\\
\label{Tr:2} &T_W \colon \tilde{K}^{*-1}(W) \to KX^*(X);\\ 
\label{Tr:3} &T_W \colon \tilde{H}^{*-1}(W) \to HX^*(X).
\end{align}
Here $\tilde{H}^*(W)$ is the reduced cohomology of $W$ and $HX^*(X)$
is the coarse cohomology of $X$. (See \cite{MR1147350}.) 
In Section~\ref{sec:coarse-homol-theor}, we give a construction of
the map (\ref{Tr:1}) which appeared in \cite[Appendix]{MR1388312}.
The map (\ref{Tr:2}) is constructed in Section~\ref{sec:coarse-k-thoery}.
The map (\ref{Tr:3}) is constructed in \cite[Section 5.3]{MR1147350}.

There exists a homomorphism $b\colon K_*(C^*(X))\to \tilde{K}_{*-1}(W)$
such that $T_W = b\circ \mu_*$.
Therefore if the transgression
map (\ref{Tr:1}) is injective, then so is the coarse assembly map for $X$.
It is also known that if (\ref{Tr:2}) or
(\ref{Tr:3}) is surjective then the coarse assembly map is
rationally injective. For details, see \cite[Appendix]{MR1388312},
\cite[(6.32)]{MR1147350} and \cite[Section 6]{MR2225040}. 
The statement that the transgression
map (\ref{Tr:3}) is surjective for some corona $W$ is 
a version of the Weinberger conjecture.
In this paper, we consider transgression maps for relatively hyperbolic groups.



Let $G$ be a finitely generated group and $\mathcal{S}$ be a finite
 generating set. We suppose that $G$ is hyperbolic relative to a finite
 family of infinite  subgroups $\famP=\{P_1,\dots,P_k\}$. 
Groves and Manning~\cite{MR2448064} defined
 the augmented space
$\Xaug$ with a properly discontinuous action of $G$ by isometries. They showed
 that $\Xaug$ is hyperbolic in the sense of Gromov.
We denote by $\partial \Xaug$ the Gromov
 boundary of $\Xaug$, which is called the Bowditch boundary of
 $(G,\famP)$. (See \cite[Definition 1.4]{Dahmani-2003}.) 
Let $(W_i,\zeta_i)$ be a corona of $P_i$. 
We blow up all parabolic points of $\partial \Xaug$
by using $W_1,\dots, W_k$ and obtain a corona 
$\partial X_\infty$ of $G$. We call $\partial X_\infty$ the
 blown-up corona of 
$(G,\famP,\{W_1,\dots, W_k\})$. See 
 Section~\ref{sec:bound-relat-hyperb} for the details of the construction.
\begin{theorem}
\label{main_theorem} 
Let $G$ be an infinite finitely generated
 group which is hyperbolic relative to a finite family of infinite subgroups
 $\famP=\{P_1,\dots,P_k\}$.
 Suppose that each subgroup $P_i$ admits a finite $P_i$-simplicial
 complex which is a universal space for proper actions. 
 For $i=1,\dots,k$, let $(W_i,\zeta_i)$ be a corona of $P_i$. 
Let $\dX_\infty$ be the blown-up corona of 
$(G,\famP, \{W_1,\dots, W_k\})$. 
\begin{enumerate}[\indent$($a$)$]
 \item 
       If ${T_{W_i}} \colon KX_*(P_i)\to \tilde{K}_{*-1}(W_i)$ is 
       an isomorphism
       for all $i= 1,\dots,k$, then so is
       ${T_{\dX_\infty}}\colon KX_*(G) \to  \tilde{K}_{*-1}(\dX_\infty)$. 
 \item 
       If $T_{W_i} \colon \tilde{K}^{*-1}(W_i)\to KX^*(P_i)$ is
       an isomorphism 
       for all $i= 1,\dots,k$, then so is
       ${T_{\dX_\infty}}\colon \tilde{K}^{*-1}(\dX_\infty) \to KX^*(G)$.
 \item 
       If $T_{W_i}\colon \tilde{H}^{*-1}(W_i)\to HX^*(P_i)$ is
       an isomorphism
       for all $i= 1,\dots,k$, then so is 
       ${T_{\dX_\infty}}\colon \tilde{H}^{*-1}(\dX_\infty) \to  HX^*(G)$.
\end{enumerate}
\end{theorem}
\begin{corollary}
\label{cor:main_cor}
 Let $G$ be an infinite finitely generated
 group which is hyperbolic relative to a finite
 family of infinite subgroups $\famP = \{P_1,\dots,P_k\}$. We suppose
 that $\famP$ satisfies all conditions in
 Theorem~\ref{th:coarse_co-assembly-map} and
 Theorem~\ref{main_theorem}. Then we have $K_*(C^*(G)) \cong
 \tilde{K}_{*-1}(\dX_\infty)$ and 
 $\tilde{K}^{*}(\dX_\infty) \cong K_*(\rsHig(G))$.
\end{corollary}
As an
 application, we give an explicit computation of the $K$-theory of the
 Roe-algebra and that of the reduced stable Higson corona of
the fundamental groups of
 closed 3-dimensional manifolds and of 
 pinched negatively curved complete Riemannian manifolds with finite volume.
See Section~\ref{sec:application}.

The organization of this paper is as follows. In
Section~\ref{sec:coarse-compactification}, we review the coarse
structure and introduce a pull-back coarse structure which plays
an essential role in the construction of coronae in 
Section~\ref{sec:bound-relat-hyperb}. We also review coronae for proper
coarse spaces. In 
Section~\ref{sec:gener-coarse-cohom}, we formulate generalized
coarse cohomology theories. In Section~\ref{sec:coarse-k-thoery-proof}, we
show that the coarse $K$-theory \cite{MR2225040} satisfies 
axioms introduced in the previous section. In
Section~\ref{sec:coarse-co-assembly}, we review the construction of the
coarse co-assembly map.
In Section~\ref{sec:coarse-cohom-hyperbolic-metric-spaces}, we 
show that the coarse co-assembly maps are isomorphisms in the case of 
proper geodesic spaces which are hyperbolic in the sense of Gromov. In
Section~\ref{sec:relat-hyperb-groups}, we review a definition of
relatively hyperbolic groups due to Groves and Manning \cite{MR2448064}
and give a proof of Theorem~\ref{th:coarse_co-assembly-map}~(\ref{item:main-th-coh}). In
Section~\ref{sec:bound-relat-hyperb}, we construct a corona of a
relatively hyperbolic group using a pull-back coarse structure. In
Section~\ref{sec:transgression-maps}, we give a proof of
Theorem~\ref{main_theorem}. In Section~\ref{sec:application}, 
we give an explicit computation for the fundamental groups of
closed 3-dimensional manifolds and of pinched negatively curved complete
Riemannian manifolds with finite volume.
In Appendix~\ref{appendix:Milnor}, we give a proof of the Milnor
exact sequence for $\sigma$-$C^*$-algebras, which we often use in the
present paper.

\section{Coarse compactification}
\label{sec:coarse-compactification}
\subsection{Coarse structure}
Here we review the coarse structure from
\cite{MR2007488} and introduce the pullback coarse structure.

Let $X$ be a set. For $E\subset X\times X$, put 
$E^{-1}:= \{(y,x): (x,y)\in E\}$ and call it the inverse of $E$. 
For $E',E''\subset X\times X$, put 
$E'\circ E'' := \{(x',x''): \exists x \in X,\, (x',x)\in E',\, (x,x'') \in
E''\}$ and call it the product of $E'$ and $E''$.


\begin{definition}
 A coarse structure on a set $X$ is a collection $\mathcal{E}$ of
 subsets of $X\times X$, called {\itshape controlled sets} for the coarse
 structure, which contains the diagonal and is closed under the
 formation of subsets, inverses, products, and finite union. A set
 equipped with a coarse structure is called a coarse space.
\end{definition}

\begin{example}
 Let $X$ be a metric space. The bounded coarse structure on $X$ is a collection
 of all subsets $E\subset X\times X$ such that 
$\sup\{d(x,x'): (x,x')\in E\}< \infty$.
\end{example}

\begin{example}
Let $G$ be a countable group. There always exists a proper left invariant
 metric $d$ on $G$. The bounded coarse
 structure on $G$ associated to $d$ does not depend on the choice of
 such a metric $d$. 
See \cite[Proposition 1.15, Example 2.13]{MR2007488}. In this paper, we
 always assume that countable
 groups are equipped with this canonical coarse structures. 
\end{example}

\begin{definition}
\label{def:bounded}
 Let $X$ be a coarse space and let $B$ be a subset of $X$. We say that
 $B$ is {\itshape bounded} if $B\times B$ is controlled.
\end{definition}

\begin{definition}
\label{def:close}
 Let $X$ be a coarse space and $S$ be a set. Two maps 
$f,g\colon S\to X$ are {\itshape close} if the set 
$\{(f(s),g(s)):s\in S\}\subset X\times X$ is controlled.
\end{definition}

\begin{definition}
\label{def:coarse-map}
 Let $X$ and $Y$ be coarse spaces, and let $f\colon X\to Y$ be a 
 map. 
\begin{enumerate}
 \item  The map $f$ is {\itshape proper} if the
 inverse image, under $f$, of each bounded subset of $Y$, is also bounded.
 \item The map $f$ is {\itshape bornologous} if for each controlled
       subset $E\subset X\times X$, the set 
$f(E)$ is a controlled subset of $Y\times Y$. Here we abbreviate
       $(f\times f)(E)$ to $f(E)$.
 \item The map $f$ is {\itshape coarse} if it is proper and bornologous.
\end{enumerate} 

The spaces $X$ and $Y$ are {\itshape coarsely equivalent} if there
       exist coarse maps $f\colon X\to Y$ and $g\colon Y\to X$ 
 such that $g\circ f$ and $f\circ g$ are close to
 the identity maps on $X$ and on $Y$, respectively. 
Such a map $f$ is called a coarse equivalence.
\end{definition}

\begin{definition}
 Let $X$ be a locally compact second countable Hausdorff space. We say
 that a coarse structure on $X$ is proper if 
\begin{enumerate}
 \item there is a controlled neighborhood of the diagonal, 
 \item every bounded subset of $X$ is relatively compact, and
 \item $X$ is coarsely connected, that is, for
       any pair of points $(x,x')\in X\times X$, the set $\{(x,x')\}$ is
       controlled. 
\end{enumerate}
\end{definition}

\begin{definition}
Let $X$ be a set and let $Y$ be a coarse space. 
Let $f\colon X\to Y$ be a map. The {\itshape pullback 
 coarse structure} on $X$ is a collection of subsets $E\subset X\times X$
 such that  $f(E)$ is a controlled subset of $Y\times Y$.
\end{definition}

\begin{proposition}
\label{prop:pull_back_coarse_str}
 Let $Y$ be a coarse space. Let $X$ be a set and let $f\colon
 X\to Y$ be a map. We equip $X$ with the pullback coarse
 structure. Then $f$ is a coarse map. 
If there exists a map 
 $g\colon Y\to X$ such
 that the composite $f\circ g$ is close to the identity, 
 then $X$ and $Y$ are coarsely equivalent. 
 If $Y$ is coarsely connected, then so is $X$.

\end{proposition}

\begin{proof}
 Let $\mathcal{E}_Y$ be a coarse structure of $Y$. The pullback coarse
 structure $\mathcal{E}_X$ is the set  $\mathcal{E}_X = \{E \subset
 X\times X: f(E)\in \mathcal{E}_Y\}$. Then
 it is trivial that $f$ is a coarse map. Suppose that there exists a map 
 $g\colon Y\to X$ such that $f\circ g$ is close to the identity.
Then a subset $F = \{(y,f\circ g(y)):y\in Y\}$ belongs to $\mathcal{E}_Y$.
Let $E\in \mathcal{E}_Y$ be a controlled set. Since
$f(g(E)) \subset F^{-1} \circ E \circ F \in \mathcal{E}_Y$, we have 
$g(E) \in \mathcal{E}_X$ . Let $B\subset X$ be a bounded set. Then 
$g^{-1}(B) \times g^{-1}(B) \subset F \circ f(B\times B) \circ F^{-1}
\in \mathcal{E}_{Y}$, so $g^{-1}(B)$ is bounded. Thus $g$ is a coarse map.
Since $f(\{(x,g\circ f(x)):x\in X\}) \subset F$, we have $g\circ f$ is
 close to the identity. 
If $Y$ is coarsely connected, then for any pair
 of points $(x,x')\in X\times X$, the set 
$\{(f(x),f(x'))\}\subset Y\times Y$ is controlled, thus so is
 $\{(x,x')\}$. Therefore $X$ is coarsely connected.
\end{proof}

\begin{definition}
\label{def:pseudocontinous}
 Let $X$ be a topological space and $Y$ be a metric space. A map 
 $f\colon X\to Y$ is {\itshape pseudocontinuous} if there exists $\epsilon>0$
 such that for any  $x\in X$, the inverse image
 $f^{-1}(B(f(x);\epsilon))$ of the closed
 ball of radius $\epsilon$ centered at $f(x)$ is a neighborhood of $x$.
\end{definition}

\begin{proposition}
\label{prop:pullback-proper} 
Let $Y$ be a proper metric space with
the bounded coarse structure. Let $X$ be a locally compact second
countable Hausdorff space. Let $f\colon X\to Y$ be a pseudocontinuous
map. We equip $X$ with the pullback coarse structure. If for any compact
set $K\subset Y$ the inverse image $f^{-1}(K)\subset X$ is
relatively compact, then $X$ is a proper coarse space.
\end{proposition}

\begin{proof}
 Fix $\epsilon>0$ satisfying the condition in
 Definition~\ref{def:pseudocontinous}. 
Set $\Delta_\epsilon = \{(x,y):d(x,y)\leq \epsilon\} \subset Y\times
 Y$. Then the pullback $f^{-1}(\Delta_\epsilon)$ is a controlled
 neighborhood of the diagonal. Let $B\subset X$ be a bounded subset, then
 $f(B)\times f(B)$ is controlled. Thus $f(B)$ is relatively compact,
 and so is $f^{-1}(f(B))$. Therefore $B$ is relatively compact. Since
 $Y$ is coarsely connected, so is $X$.
\end{proof}
The following is a typical example of the
pullback coarse structure.
\begin{proposition}
\label{prop:nerve-coarse-str}
 Let $X$ be a proper metric space. Let $\U$ be a locally finite cover
 of $X$ such that any element of $\U$ has uniformly bounded diameter.
 Then (a geometric realization of) the nerve complex $\abs{\U}$ has a
 canonical coarse structure which is proper and coarsely equivalent to $X$.
\end{proposition}
\begin{proof}
Since $X$ is a proper metric space and 
$\U$ is locally finite, $\abs{\U}$ is locally compact second countable
 Hausdorff space.
For each element $U\in \U$, we choose a point $x(U) \in U$. 
For each point $p\in \abs{\mathcal{U}}$, we choose $U_p\in \mathcal{U}$ 
such that $p\in \mathrm{st}{\,U_p}$, 
where $\mathrm{st}{\, U_p}$ denotes the star of $U_p$. 
Then we define a map $f\colon \abs{\U} \to X$ by $f(p) = x(U_p)$. 
Since $\U$ is locally finite, the pullback $f^{-1}(K)$ of any compact set
 $K\subset X$ is relatively compact.
Since each $U\in \U$
 has uniformly bounded diameter, $f$ is pseudocontinuous. 
Let $g\colon X\to \abs{U}$ be a continuous map induced by a
 partition of unity. It is easy to see that $f\circ g$ is
 close to the identity. Thus the assertion
 follows from Proposition~\ref{prop:pull_back_coarse_str} and
 \ref{prop:pullback-proper}.
\end{proof}

\subsection{Higson compactification}
\label{sec:higs-comp}
Here we recall the definitions of the Higson compactification and coarse
compactifications. For details, see \cite{MR2007488} and
\cite{MR1147350}.
\begin{definition}
\label{def:Higson_function}
 Let $X$ be a proper coarse space and let $V$ be a normed
 space. Let  $f\colon X\to V$ be a
 bounded continuous function. We denote by $\grad f$ the function
\[
  \grad f(x,y) = f(y) - f(x) \colon X\times X \to V.
\]
We say that $f$ is a Higson function, or, of vanishing variation, if for
each controlled set $E$, the restriction of $\grad f$ to $E$ vanishes at
infinity, that is, for any $\epsilon>0$, there exists a bounded subset
$B$ such that for any $(x,y)\in E \setminus B\times B$, we have
$\norm{\grad f(x,y)}<\epsilon$.
\end{definition}

The bounded continuous $\C$-valued Higson functions on a proper coarse
space $X$ form a unital
$C^*$-subalgebra of bounded continuous functions on $X$, which we
denote $C_h(X)$. By the Gelfand-Naimark theory, $C_h(X)$ is isomorphic
to a $C^*$-algebra of continuous functions on a compact Hausdorff space.

\begin{definition}
\label{def:Higson-compactification}
 The compactification $hX$ of $X$ characterized by the property 
$C(hX) = C_h(X)$ is called the Higson compactification. Its boundary 
$hX\setminus X$ is denoted $\nu X$, and is called the Higson corona 
of $X$.
\end{definition}

The assignment $X\mapsto \nu X$ is a functor from the coarse category to
the category of compact Hausdorff spaces. 
For details, see \cite[Section 2.3]{MR2007488} or 
\cite[Section 5.1]{MR1147350}.
\begin{proposition}[Dranishnikov]
\label{prop:Dra-embedding}
Let $X$ and $Y$ be proper metric spaces and let $f\colon X \to Y$
 be a coarse embedding, that is, a coarse 
equivalence to the image. Then the induced map 
$\nu f\colon \nu X\to \nu Y$ is an embedding, thus we can regard $\nu X$ as
 a subspace of $\nu Y$.
\end{proposition}
\begin{proof}
The proposition follows immediately from \cite[Theorem 1.4]{MR1607744}.
\end{proof}

\begin{definition}
Let $X$ be a proper coarse space.
 A corona of $X$ is a pair $(W,\zeta)$ of a compact metrizable space
 $W$ and a continuous map $\zeta \colon \nu X\to W$.
\end{definition}



Let $X$ be a proper coarse space. Let $(W,\zeta)$ be a corona of $X$.
We consider the disjoint union $X \sqcup W$. We equip 
$X \sqcup W$ with the final topology with respect to the map 
$\mathrm{id}\sqcup \zeta \colon hX \to X \sqcup W$,
which we denote by $\bar{\zeta}$. 
Let $X\cup_\zeta W$ denote the space
$X\sqcup W$ with this topology. By the construction, we see that
$X\cup_\zeta W$ is compact.

Next, we construct a compact Hausdorff space using functional analysis. 
The continuous map $\zeta$ induces a homomorphism 
$\zeta^*\colon C(W) \to C(\nu X)$. Then the image $\zeta^*(C(W))$ is a
$C^*$-subalgebra of $C(\nu X)$. 
Let 
\[
 \pi \colon C_h(X) \to C_h(X)/C_0(X) \cong C(\nu X)
\] 
be the quotient map. 
Then the pullback $\pi^{-1}(\zeta^*(C(W)))$ is a $C^*$-subalgebra of
$C_h(X)$. Set 
$A= \{(f,g)\in \pi^{-1}(\zeta^*(C(W)))\oplus C(W)\colon 
\pi(f) =\zeta^*(g)\}$. Then $A$ is a unital commutative $C^*$-algebra
which contain $C_0(X)$ as an ideal. 
By the Gelfand-Naimark theory, there exists a compact Hausdorff space
$Z$ and an embedding $i\colon X\to Z$ such that 
$C(Z) \cong A$. We identify $X$ with $i(X)$. 

\begin{proposition}
\label{prop:coarse_compt-is-Hausdorff}
These two spaces $X\cup_\zeta W$ and $Z$ are homeomorphic. 
Especially, $X\cup_\zeta W$ is a compact metrizable space. 
If $\zeta$ is surjective, $X$ is dense in $X\cup_\zeta W$ and thus we 
call $X\cup_\zeta W$ a coarse compactification of $X$.
We abbreviate 
$X\cup_\zeta W$ to $X\cup W$ for simplicity. 
\end{proposition}

\begin{proof}
Let $A$ be a $C^*$-algebra defined in the above. The inclusion 
$C_0(X) \hookrightarrow A$ is given by $f\mapsto (f,0)$. 
We also have a surjection $A\to C(W),\, (f,g)\mapsto g$.
We consider the following diagram with two short exact sequences
\begin{align*}
 \xymatrix{
0 \ar[r] &C_0(X) \ar[r] &C(hX) 
 \ar[r] & C(\nu X) \ar[r] & 0\\
0 \ar[r] &C_0(X) \ar[r] \ar@{=}[u] &C(Z) 
 \ar[r] \ar[u] & C(W) \ar[r] \ar[u] & 0.
}
\end{align*}
Since $C(W)$ and $C_0(X)$ are separable, so is $C(Z)$. Thus $Z$ is
 metrizable. The surjection $C(Z) \to C(W)$ induces an
 embedding $W\to Z$, so we identify $W$ with its image in
 $Z$. Thus $Z$ can be decomposed as $Z = X \cup W$. 
 Let $\varphi \colon X\cup_\zeta W \to Z$ be the canonical
 bijection. Then we have a commutative diagram
\begin{align*}
 \xymatrix{
hX \ar[d]_{\zeta} \ar[dr] &\\
X\cup_\zeta W \ar[r]^\varphi & Z.
}
\end{align*}
Since the map $hX \to Z$ is continuous, so is
$\varphi$. Therefore $\varphi$ is homeomorphism.
\end{proof}
The following notion is useful in the study of
proper metric spaces and their coronae from the view point of the
algebraic topology.
\begin{definition}
\label{def:cover}
Let $X$ and $Y$ be proper metric spaces and let $(W,\zeta)$ and
 $(Z,\xi)$ be respectively coronae of $X$ and $Y$.
 Let $f\colon X\to Y$ be a coarse map and let $\eta \colon W\to Z$ be a
 continuous map. We say that $f$ {\itshape covers} $\eta$ if there exists a
 discrete subset $X'\subset X$ such that
 the inclusion is a coarse
 equivalence and the restriction $f|_{X'}$ extends to a continuous map 
 $f\cup \eta\colon X'\cup W\rightarrow Y\cup Z$.
\end{definition}

\begin{remark}
 In the above setting, $f$ covers $\eta$ if and only if the following
 diagram is commutative
\begin{align*}
 \xymatrix{
\nu X \ar[r]^{\nu f}  \ar[d]_\zeta &\nu Y \ar[d]_\xi\\
 W \ar[r]^{\eta}  & Z.
}
\end{align*}
\end{remark}


In the rest of the paper, whenever we consider a corona $(W,\zeta)$
of a proper metric space $X$, we assume that $X$ is non-compact. In
particular, neither $\nu X$ nor $W$ is empty.

\section{Generalized coarse cohomology theory}
\label{sec:gener-coarse-cohom}
\subsection{Axiom}
\label{sec:axiom}
The coarse category is a category whose objects are proper metric spaces
and whose morphisms are close classes of coarse maps. The coarse
cohomology \cite{MR1147350}, the coarse $K$-theory \cite{MR2225040} and
the $K$-theory of the reduced stable Higson corona \cite{MR2225040} can be
regarded as cohomology theories on the coarse category.  In this
section, we introduce a generalized coarse cohomology theory.

The following notion was introduced in \cite{MR1219916} to state the
Mayer-Vietoris principle for the coarse cohomology and the $K$-theory of
the Roe algebra. Let $X$ be a metric space and $A\subset X$ be a
subspace. For $R>0$, we denote by $\Pen(A;R)$ the $R$-neighborhood 
$\{x \in X: d(x,A)\leq R\}$ of $A$.
\begin{definition}
 Let $X$ be a proper metric space, and let $A$ and $B$ be closed
 subspaces with $X = A\cup B$. We say that 
 $X= A\cup B$ is an {\itshape $\omega$-excisive decomposition}, 
 if for each $R> 0$ there exists some $S> 0$ such that 
\[
 \Pen(A;R)\cap \Pen(B;R) \subset \Pen(A\cap B; S).
\]
\end{definition}

Higson-Roe~\cite{MR1388312} introduced a
 notion of coarse homotopy. After that, they gave an
 alternative definition of coarse homotopy, which is a variant of Lipschitz
homotopy. (For Lipschitz homotopy, see \cite[1.$C_3$]{MR919829},
\cite[Definition 4.1]{MR1344138} and \cite[Definition 11.1]{MR1451755}.)
Our definition is based on \cite[Section 11]{MR1817560} 
and \cite[Definition 3.9]{WillettThesis}.

\begin{definition}
 Let $f,g\colon X\to Y$ be coarse maps between proper metric
 spaces. We say that they are
 {\itshape coarsely homotopic} if there exists a metric subspace 
$Z = \{(x,t):1\leq t\leq T_x\}$ 
of $X\times \N$ and a coarse map
 $h\colon Z\to Y$, such that
\begin{enumerate}
 \item the map $x\mapsto T_x$ is bornologous,
 \item $h(x,1) = f(x)$, and
 \item $h(x,T_x) = g(x)$.
\end{enumerate}
Here $\N$ is a set of positive integers 
and we equip $X\times \N$ with the $l_1$-metric, that is, 
$d_{X\times \N}((x,n),(y,m)):= d_X(x,y)+ \abs{n-m}$ for 
$(x,n),(y,m)\in X\times \N$, where $d_X$ is the metric on $X$.
\end{definition}
Coarse homotopy is then an equivalence relation on coarse maps.

\begin{definition}
\label{def:generalized-coarse-cohomology}
 A {\itshape generalized coarse cohomology theory} is a contravariant functor
$MX^* = \{MX^p\}_{p \in \Z}$ from the coarse category to the category of
 $\Z$-graded Abelian groups, such that
\begin{enumerate}[(i)]
 \item \label{axiom:product-with-half-line}
       for a proper metric space $\mX$, we have
       $MX^*(\mX\times \N) = 0$, and
 \item \label{axiom:Mayer-Vietoris}
       if $\mX = A\cup B$ is an $\omega$-excisive decomposition, there
       exists a functorial long exact sequence, 
       called a Mayer-Vietoris exact sequence, 
\[
 \dots \to MX^p(\mX) \to MX^p(A) \oplus MX^p(B) 
       \to MX^p(A\cap B) \to MX^{p+1}(\mX) \to \cdots.
\]
\newcounter{axiomnum}
\setcounter{axiomnum}{\value{enumi}}
\end{enumerate}
\end{definition}

The following notion of coarsely flasque spaces is based on 
\cite[Definition 3.6]{WillettThesis}.
\begin{lemma}
\label{lem:flasque}
Let $MX^*$ be a generalized coarse cohomology theory.
Let $\mX$ be a space with a proper metric $d$. 
Suppose that $\mX$ is coarsely flasque, 
that is, 
there exists a coarse map $\phi:\mX\to \mX$ such that 
\begin{enumerate}
\item $\phi$ is close to the identity;
\item for any bounded subset $K\subset Y$, there exists $N_K\in \N$
such that for any $n\ge N_K$, $\phi^n(\mX)\cap K=\emptyset$;   
 \item for all $R>0$, there exists $S>0$ such that for all $n\in \N$ and
       all $x,y \in \mX$ with $d(x,y)<R$, we have 
       $(d(\phi^n(x),\phi^n(y)))<S$.
\end{enumerate}
Then $MX^*(\mX)=0$.
\end{lemma}

\begin{proof}
 We define a coarse map $\Phi\colon \mX\times \N \to \mX$ as
$\Phi(x,n) = \phi^n(x)$.
 Then we have a commutative diagram
\begin{align*}
 \xymatrix{
 \mX \ar@{^{(}->}[dr]\ar[rr]^\phi&&\mX\\
 &\mX\times \N \ar[ur]_{\Phi}&.
}
\end{align*} Here $\mX\hookrightarrow \mX\times\N$ is the inclusion into 
$\mX\times \{1\}$. By axiom (\ref{axiom:product-with-half-line}), the induced
map $\phi^*\colon MX^*(\mX) \to MX^*(\mX)$ factors through zero.
Since $\phi$ is close to the identity, $MX^*(\mX) = 0$.
\end{proof}

The following coarse homotopy invariance follows from a standard
argument using Mayer-Vietoris axiom~(\ref{axiom:Mayer-Vietoris}) and 
Lemma~\ref{lem:flasque}. (See \cite[Proposition 12.4.12]{MR1817560} 
and \cite[Theorem 4.3.12.]{WillettThesis}).
\begin{proposition}
 If two coarse maps $f,g\colon X\to Y$ are coarsely homotopic,
 the induced maps $f^*$ and $g^*$ are equal.
\end{proposition}

The anti-\v{C}ech system is introduced in \cite[Section 3]{MR1147350}
to relate the coarse cohomology to the \v{C}ech cohomology.
It is also used in \cite{MR1388312} to formulate a coarse
homology theory.
\begin{definition}
\label{def:anti-Cech-sys}
 Let $X$ be a metric space. Let $\U(1),\U(2),\dots$ be a sequence of
 locally finite covers of $X$. We say that they form an anti-\v{C}ech
 system if there exists a sequence of real numbers 
$R_n\to \infty$ such that for all $n$, 
\begin{enumerate}
 \item each set $U\in \U(n)$ has diameter less than or 
equal to $R_n$, and
 \item the covering $\U(n+1)$ has a Lebesgue number $\delta_{n+1}$
       greater than or equal to $R_n$, that is, any set of diameter less
       than or equal to $\delta_{n+1}$ is contained in some element of
       $\U(n+1)$.
\end{enumerate}
\end{definition}
These conditions imply that for each
$n$, 
there exists a map
$\varphi_{n}\colon \U(n)\to \U(n+1)$ such that 
$U\subset \varphi_n(U)$ for all $U\in \U(n)$. We call $\varphi_{n}$ a
coarsening map. We remark that this map is called a refining map in the
context of \v{C}ech cohomology theory.
A coarsening map $\varphi_{n}$ induces a proper
simplicial map $\abs{\U(n)}\rightarrow \abs{\U(n+1)}$ of the nerve
complexes, which we also denote by the same symbol $\varphi_{n}$ and
also call a coarsening map. 
In this paper, we use the same notation for the nerve of an anti-\v{C}ech
system, and its geometric realization.

Now we recall the definition of a generalized cohomology theory on the
category of locally compact and second countable Hausdorff spaces, which
we abbreviate to LCSH. (See \cite[Section 7.1]{MR1817560} for LCSH.)
A generalized cohomology theory on LCSH  is a contravariant functor
$M^*=\{M^p\}$ from LCSH to the category of $\Z$-graded Abelian groups
such that 
\begin{enumerate}
 \item $M^*$ is a homotopy functor, and 
 \item if $W\subset X$ is a closed subset, there is a functorial long
       exact sequence
\[
 \dots \to M^p(X\setminus W) \to M^p(X)
       \to M^p(W) \xrightarrow{\partial} M^{p+1}(X\setminus W) \to \cdots.
\]
\end{enumerate}
Examples of such cohomology theories are $K$-theory $K^*(-)$ and 
the Alexander-Spanier cohomology with compact supports $H_c^*(-)$. These
cohomology theories satisfy the continuity property
\begin{enumerate}
\setcounter{enumi}{2}
 \item for a
projective limit $X=\varprojlim X_n$ of locally compact second countable
Hausdorff spaces, we have $M^*(X)\cong \varinjlim M^*(X_n)$.
\end{enumerate}

Let $W$ be a compact second countable Hausdorff space. 
Then the constant map $\pi_W\colon W\to \{*\}$ is proper, where $\{*\}$
is a one point space. The reduced $M$-cohomology of $W$, denoted by
$\tilde{M}^*(W)$, is defined as the cokernel of $\pi_W^*$.

Let $X$ be a proper metric space and let $(W,\zeta)$ be a corona of $X$. 
Let $\partial \colon M^p(W) \rightarrow M^{p+1}(X)$ be a boundary
homomorphism of the long exact sequence for $W\subset X\cup_\zeta W$. 
Let $\pi_W\colon W\rightarrow \{*\}$ be a constant map. Since $\pi_W$
factors through $X\cup_\zeta W \rightarrow \{*\}$, the image
$\pi_W^*(M^p(\{*\}))$ lies on the kernel of $\partial$.
Thus we have a boundary homomorphism 
$\partial \colon \tilde{M}^p(W)\rightarrow M^{p+1}(X)$.

\begin{definition}
\label{def:coarsening-of-cohomology-theory} Let 
$M^* = \{M^p\}_{p\in \Z}$ 
 be a generalized cohomology theory on locally compact and second
 countable Hausdorff spaces. We say that a generalized coarse cohomology
 theory $MX^*$ is a {\itshape coarsening} of $M^*$ if $MX^*$ satisfies the
 following:
\begin{enumerate}[(i)]
 \setcounter{enumi}{\value{axiomnum}}
  \item \label{axiom:character-map}
	For a proper metric space $X$, there exists a {\itshape character map}
	$c\colon MX^*(X)\to M^*(X)$, which is an isomorphism if
	$X$ is uniformly contractible 
	and has bounded geometry. It is
	compatible with Mayer-Vietoris exact sequences of $MX^*$ and
	$M^*$ for $\omega$-excisive decompositions. 
 \item \label{axiom:Milnor}
       Let $\{\mathcal{U}_n\}$ be an anti-\v{C}ech system of a proper metric
 space $X$. There exists a functorial short
 exact sequence
\[
 0 \to \limone M^{q-1}(\abs{\mathcal{U}_n}) 
 \to MX^q(X) \xrightarrow{\theta} \varprojlim
 M^q(\abs{\mathcal{U}_n}) \to 0.
\]
Moreover, the composite of $\theta$ and a canonical map 
$\lambda\colon \varprojlim M^q(\abs{\mathcal{U}_n}) \to M^q(X)$ is equal
       to the character map, where $\lambda$ is given by a partition of
       unity. We call this a Milnor exact sequence.
 \item \label{axiom:transgression-map}
       Let $(W,\zeta)$ be a corona of $X$. 
       Then there exists a {\itshape transgression} map 
       $T_W\colon \tilde{M}^{q-1}(W)\to MX^q(X)$ such that 
       $c\circ T_W = \partial$, 
       here $\partial \colon \tilde{M}^{q-1}(W) \to M^q(X)$ is
       the boundary homomorphism.
       The transgression map is natural in the following sense.
       For proper metric spaces $X$ and $Y$, and for coronae $(W,\zeta)$ and
       $(Z,\xi)$ respectively of $X$ and $Y$, if a coarse map 
       $f\colon X\to Y$ covers a continuous map $\eta\colon W\rightarrow
       Z$, then the following is commutative.
\begin{align*}
  \xymatrix{
 \tilde{M}^{q-1}(Z) \ar[d]^{T_Z}\ar[r]^{\eta^*}& \tilde{M}^{q-1}(W) \ar[d]^{T_W}\\
 MX^q(Y) \ar[r]^{f^*}& MX^q(X).
}
\end{align*}
\end{enumerate}
\end{definition}

\begin{proposition}
\label{prop:coarse_cohomology}
 The coarse cohomology $HX^*(-)$, the coarse $K$-theory $KX^*(-)$ and
 the $K$-theory of the reduced stable Higson corona $K_*(\rsHig(-))$ are
 generalized coarse cohomology theories. Especially, $KX^*(-)$ and
 $HX^*(-)$  are respectively the coarsening of the $K$-theory and 
 the Alexander-Spanier cohomology with compact supports .
\end{proposition}
\begin{proof}
 The statements for $HX^*$ are proved in \cite{MR1147350}, those for 
$K_*(\rsHig(-))$ are proved in \cite{MR2225040} and
 \cite{WillettThesis}. 
See Proposition~\ref{prop:K-theory-of-sHig-is-cohomology}.
The statements for $KX^*$ are proved in Section~\ref{sec:coarse-k-thoery}.
\end{proof}

\subsection{Coarse homology theories}
\label{sec:coarse-homol-theor}
Generalized coarse homology theories are formulated similarly to 
Definition~\ref{def:generalized-coarse-cohomology}, but we omit the
detail. We remark that for a generalized homology theory $M_*$ on LCSH,
we have a generalized coarse homology theory $MX_*$ by defining 
$MX_*(X) := \varinjlim M_*(\abs{\U(j)})$ where $X$ is a proper metric
space and $\{\U(j)\}_{j\in \N}$ is an anti-\v{C}ech system of $X$. 
(See \cite[Section 2]{MR1388312}.) We say that $MX_*$ is a coarsening of
$M_*$. Using a partition of unity, we can define the coarsening map 
$c\colon M_*(X)\to MX_*(X)$.
 If $X$ is uniformly contractible and has bounded geometry, the
 coarsening map $c$ is an isomorphism. Emerson-Mayer proved a similar
 statement for coarse $K$-theory. (See \cite[Theorem 4.8]{MR2225040}.)
 Their proof also works for $MX_*$.
We remark that this statement is first proved in~\cite[Proposition
3.8]{MR1388312} under an additional assumption that $X$ is a simplicial
complex with a spherical metric. 

The transgression map is constructed as follows.
Let $X$ be a proper metric space and let $(W,\zeta)$ be a corona of $X$.
Let $\{\U_n\}_{n\in \N}$ be an anti-\v{C}ech system of $X$.
Since the nerve complex $\abs{\U_n}$ is coarsely equivalent to $X$
(Proposition \ref{prop:nerve-coarse-str}),
the pair
$(W,\zeta)$ is also a corona of $\abs{\U_n}$ and we obtain a
compact space
$\abs{\U_n}\cup W$. A long exact sequence
(\cite[Definition 7.1.1]{MR1817560}) for $W\subset \abs{\U_n}\cup W$ 
defines the boundary homomorphism
$\partial \colon M_*(\abs{\U_n}) \to \tilde{M}_{*-1}(W)$. 
Here
$\tilde{M}_*(W)$ is the reduced $M$-homology of $W$ defined as the
kernel of ${\pi_W}_*$, where $\pi_W\colon Y\rightarrow \{*\}$ is a constant map.
By taking the
inductive limit, we obtain 
${T_W}\colon MX_*(X) \to \tilde{M}_{*-1}(W)$. From the construction, it is
easy to see that the transgression map is natural in the obvious sense.

The $K$-theory of the Roe-algebra, the coarse $K$-homology are generalized
coarse homology theories and the coarse $K$-homology is the coarsening of
the $K$-homology. See \cite{MR1388312}, \cite{MR1219916} and \cite{MR1817560}.
\section{The coarse $K$-theory}
\label{sec:coarse-k-thoery}
\subsection{The coarse $K$-theory}
\label{sec:coarse-k-thoery-proof}
In this section we see that 
the coarse $K$-theory $KX^*(-)$ 
is a generalized coarse cohomology theory 
and is the coarsening of the $K$-theory $K^*(-)$
in the sense of the previous section. 
Originally, $KX^*(-)$ is defined and studied by Emerson-Meyer 
\cite[Section 4]{MR2225040}. 
We introduce a definition of $KX^*(-)$ by a slightly different manner, 
but we confirm that they are compatible. 
The original definition uses the Rips complex, while ours uses
the anti-\v{C}ech system, which is more flexible and essentially used in the
proof of Proposition~\ref{prop:weak-coarsening}.

Let $X$ be a space with a proper metric $d$. 
Suppose that $\{\U(k)\}_{k\in\N}$ is an anti-\v{C}ech system of $X$
with uniformly bounded diameter $R_k\to \infty$  and 
Lebesgue numbers $\delta_k\geq  R_{k-1}$ of $\U(k)$.


For each $k\in \N$, we fix a coarsening map 
$\psi_{k,k+1}\colon\abs{\U(k)} \to \abs{\U(k+1)}$.
We put $\psi_{k,l}:=\psi_{l-1,l}\circ\cdots \circ\psi_{k,k+1}$ for each $k\in \N$
and $l\in \N$ with $k\le l-1$ and we also call them coarsening maps. 
We denote the inductive limit by $\mathcal X$, 
which depends on choice of $\psi_{k,k+1}$. 
Also we denote the canonical map $\abs{\U(k)}\to \mathcal X$ 
by $\psi_{k,\infty}$ for each $k\in\N$.
We put 
\[
C_0(\mathcal X):=\{f:\mathcal X\to \mathbb C 
\left| \ f\circ \psi_{k,\infty}\in C_0(\abs{\U(k)})\text{ for any }k\in\N \right.\}
\]
and we identify it with the projective limit of $\{C_0(\abs{\U(k)})\}_{k\in \N}$. 
This is a $\sigma$-$C^\ast$-algebra. 
Now we define $KX^*(X)$ as $RK_*(C_0(\mathcal X))$.
Here $RK_*(-)$ is a representable $K$-theory of 
$\sigma$-$C^\ast$-algebras \cite{MR1050490}.
We abbreviate $RK_*(C_0(\mathcal X))$ to $RK^*(\mathcal X)$. 

We remark that by Phillips \cite{MR1050490}, there exists an exact
sequence, called a Milnor exact sequence,
\begin{align}
\label{eq:Milnor}
 0\to \limone K_{p+1}(C_0(\abs{\U(k)})) \to RK^p(\mathcal X) \to
\varprojlim K_p(C_0(\abs{\U(k)})) \to  0.
\end{align}
See also Appendix~\ref{appendix:Milnor}.

\begin{lemma}
\label{lem:injective}
 Under the above setting, there exists an
 anti-\v{C}ech system $\{\U'(k)\}$ such that a coarsening map
$\psi'_k\colon \U'(k)\to \U'(k+1)$ is injective for each $k\in \N$ and 
$RK^*(\mathcal X)\cong RK^*(\varinjlim \U'(k))$.
\end{lemma}

\begin{proof}
We take a copy $\U_i(k)$ of $\U(k)$ parameterized by $i\in\N$.
Then $\bigcup_{i\in\N} \U_i(k)$ is a cover of $X$, but it is not locally finite. 
The identification between $\U_i(k)$ and $\U(k)$ define the surjection 
$P_k:\bigcup_{i\in\N} \U_i(k)\to \U(k)$.  
Then we can take an anti-\v{C}ech system $\{\U'(k)\}$ of $X$
and proper injective simplicial 
map $\psi'_{k,k+1}:\abs{\U'(k)}\to \abs{\U'(k+1)}$
satisfying 
$\U'(1)=\U_1(1)$, $\U_1(k)\subset \U'(k)\subset \bigcup_{i\in\N}\U_i(k)$ 
and the following commutative diagram:
\[
\xymatrix{
\U'(1)\ar[r]^{\psi'_{1,2}}\ar[d]^{p_1}
&\U'(2)\ar[r]^{\psi'_{2,3}}\ar[d]^{p_2}&\U'(3)\ar[r]^{\psi'_{3,4}}\ar[d]^{p_3}&\cdots\\
\U(1)\ar[r]^{\psi_{1,2}}&\U(2)\ar[r]^{\psi_{2,3}}&\U(3)\ar[r]^{\psi_{3,4}}&\cdots,
}
\]
where $p_k$ is a proper surjective simplicial map 
induced by $P_k$ of the restriction on $\abs{\U'(k)}$. 
For each $k$, we choose a section $e_k\colon \U(k)\to \U'(k)$ of $p_k$.
Then we have the following commutative diagram:
\[
\xymatrix{
\U'(1)\ar[r]^{\psi'_{1,2}}\ar[d]^{p_1}&\U'(2)\ar[r]^{\psi'_{2,3}}\ar[d]^{p_2}
&\U'(3)\ar[r]^{\psi'_{3,4}}\ar[d]^{p_3}&\cdots\\
\U(1)\ar[r]^{\psi_{1,2}}\ar[d]^{e_1}&\U(2)\ar[r]^{\psi_{2,3}}\ar[d]^{e_2}
&\U(3)\ar[r]^{\psi_{3,4}}\ar[d]^{e_3}&\cdots\\
\U'(1)\ar[r]^{e_2 \circ \psi_{1,2}\circ p_1}\ar[d]^{p_1}
&\U'(2)\ar[r]^{e_3 \circ \psi_{2,3}\circ p_2}\ar[d]^{p_2}
&\U'(3)\ar[r]^{e_4 \circ \psi_{3,4}\circ p_3}\ar[d]^{p_3}&\cdots\\
\U(1)\ar[r]^{\psi_{1,2}}&\U(2)\ar[r]^{\psi_{2,3}}&\U(3)\ar[r]^{\psi_{3,4}}&\cdots.
}
\]
Note that the inductive limits of
the second line and the forth line are $\mathcal X$.  We denote by 
$\mathcal{X}''$ and $\mathcal{X}'''$, respectively,
the inductive limits of the first line and the third line.
Since every $p_k\circ e_k$ are identity maps, 
$(\varinjlim e_k)^*\colon RK^*(\mathcal{X'''})\to RK^*(\mathcal{X})$ 
is surjective.
The Milnor exact sequence and its
functoriality imply the following commutative diagram:
\[
\xymatrix{
0\ar[r]
&\limone K^{*-1}(\abs{\U'(k)})
\ar[r] 
&RK^*(\mathcal X'')\ar[r]
&\varprojlim K^*(\abs{\U'(k)})\ar[r]
&0\\
0\ar[r]
&\limone K^{*-1}(\abs{\U'(k)}))\ar[r]
\ar[u]_{\limone(e_k\circ p_k)^{*-1}}
&RK^*(\mathcal X''')\ar[r]
\ar[u]_{(\varinjlim e_k\circ \varinjlim p_k)^*}
&\varprojlim K^*(\abs{\U'(k)})\ar[r]
\ar[u]_{\varprojlim (e_k\circ p_k)^*}
&0.}
\]
Since $e_k\circ p_k$ is contiguous to the identity map, 
$(\varinjlim e_k\circ \varinjlim p_k)^*$
is an isomorphism by the five lemma, and thus $(\varinjlim e_k)^*$ is injective.
Hence $(\varinjlim p_k)^*\colon RK^*(\mathcal{X})\to
 RK^*(\mathcal{X}'')$ is an isomorphism.
\end{proof}

\begin{proposition}\label{well-def}
$KX^*(X)$ is well-defined, that is, 
this is independent of the choice of the anti-\v{C}ech system 
$\{\U(k)\}_{k\in\N}$ and the coarsening maps $\{\psi_{k,l}\}_{k\le l}$.  
\end{proposition}
\begin{proof}
Let $\{\U(k)\}_{k\in\N}$ be an anti-\v{C}ech system and let 
$\{\psi_{k,l}\}_{k\le l}$ be coarsening maps. 
By Lemma~\ref{lem:injective}, we can assume that $\psi_{k,l}$ is
 injective. We denote by $\mathcal{X}$ the injective limit of $\{\U(k)\}$.

We compare $\{\U(k)\}$ with a special kind of an anti-\v{C}ech system 
of $X$ defined as follows.  We take a subset
$Z$ of $X$ and a constant $C > 0$ such that $\Pen(Z,C)=X$ and 
$d(x,y)> 1$ for any $x,y\in Z$ with $x\neq y$. 
The existence of such a subset follows from Zorn's lemma. 
(See \cite[Lemma 3.15]{MR1147350}.)
We call $Z$
a $C$-dense uniformly discrete subset of $X$.
For each $k\in \N$, put 
$\U_{Z,C}(k):=\left\{\Pen(z,(k+1)C) \subset X \left| z\in Z\right\}\right.$ 
which is a locally finite cover of $X$ since $X$ is proper.
For each $k\in \N$, diameter of any element of $\U_{Z,C}(k)$ is at most
 $2(k+1)C$ and the Lebesgue number of $\U_{Z,C}(k)$ is at least $kC$.
Hence $\{\U_{Z,C}(k)\}_{k\in\N}$ is an anti-\v{C}ech system of $X$.  
We have a proper simplicial map 
$\iota_{k,l}:\abs{\U_{Z,C}(k)}\to \abs{\U_{Z,C}(l)}$ induced by 
$\U_{Z,C}(k)\ni \Pen(z,(k+1)C)\mapsto \Pen(z,(l+1)C)\in\U_{Z,C}(l)$ 
for each $k\in \N$ and $l\in \N$ with $k\le l$. 
We denote the inductive limit by $\mathcal X_{Z,C}$. 
Also we denote the induced map $\abs{\U_{Z,C}(k)}\to \mathcal X_{Z,C}$ 
by $\iota_{k,\infty}$ for each $k\in\N$.
Note that $\iota_{k,l}$ is injective for any $k\in \N$ and $l\in \N\cup\{\infty\}$
with $k\le l$. 

We prove that $RK^*(\mathcal X)$ and $RK^*(\mathcal X_{Z,C})$
are canonically isomorphic. Then we have the desired conclusion. 
We take an increasing sequence $\{k_j\in\N\}$ such that for each $j$, 
the cover $\U(j)$ is an refinement of $\U_{Z,C}(k_j)$. Then for each
 $j\in\N$, we can choose an
coarsening map $f_j\colon \U(j)\rightarrow \U_{Z,C}(k_j)$ 
such that the following diagram: 
\[
\xymatrix{
\abs{\U(1)}\ar[r]^{\psi_{1,2}} \ar[d]^{f_1} &\abs{\U(2)} \ar[r]^{\psi_{2,3}}\ar[d]^{f_2} & 
 \cdots\\ 
\abs{\U_{Z,C}(k_1)}\ar[r]^{\iota_{k_1,k_2}} &\abs{\U_{Z,C}(k_2)} \ar[r]^{\iota_{k_2,k_3}}& 
 \cdots 
}
\]
is commutative without arranging any maps in both horizontal lines. 

Next, we take an increasing sequence $\{k'_j\in\N\}$ such that
for each $j$, $\U_{Z,C}(k_j)$ and $\U(k'_j)$ are respectively
 refinement of $\U(k'_j)$ and $\U(k'_{j+1})$.
Then we can choose coarsening maps 
$g_j\colon \U_{Z,C}(k_j) \to \U(k'_j)$ and 
$\psi'_{k'_j,k'_{j+1}}\colon \U(k'_j)\to \U(k'_{j+1})$ 
such that the following diagram: 
\[
\xymatrix{
\abs{\U_{Z,C}(k_1)}\ar[r]^{\iota_{k_1,k_2}} \ar[d]^{g_1} 
&\abs{\U_{Z,C}(k_2)} \ar[r]^{\iota_{k_2,k_3}}\ar[d]^{g_2} & 
 \cdots \\
\abs{\U(k'_1)}\ar[r]^{\psi'_{k'_1,k'_2}} &\abs{\U(k'_2)} \ar[r]^{\psi'_{k'_2,k'_3}}& 
 \cdots 
}
\]
is commutative.
We note that $\psi'_{k'_j,k'_{j+1}}$
is contiguous to $\psi_{k'_j,k'_{j+1}}$ and that 
$g_j\circ f_j$ is contiguous to $\psi_{j,k'_j}$.
We denote by $\mathcal X'$ the inductive limit of the second horizontal line.
We remark that there are no canonical map from $\mathcal X'$ to
 $\mathcal X$ in general.

Again, we take an increasing sequence $\{k''_j\in\N\}$ such that for
 each $j$, covers $\U(k'_j)$ and $\U_{Z,C}(k''_j)$ are respectively refinements
 of $\U_{Z,C}(k''_j)$ and $\U_{Z,C}(k''_{j+1})$. Then we can choose
 coarsening maps $h_j:\U(k'_j)\to \U_{Z,C}(k''_j)$
and $\iota'_{k''_j,k''_{j+1}}:\U_{Z,C}(k''_j)\to \U_{Z,C}(k''_{j+1})$ 
such that the following diagram: 
\[
\xymatrix{
\abs{\U(k'_1)}\ar[r]^{\psi'_{k'_1,k'_2}} \ar[d]^{h_1}
 &\abs{\U(k'_2)} \ar[r]^{\psi'_{k'_2,k'_3}}\ar[d]^{h_2} &  
 \cdots\\ 
\abs{\U_{Z,C}(k''_1)}\ar[r]^{\iota'_{k''_1,k''_2}}
 &\abs{\U_{Z,C}(k''_2)} \ar[r]^{\iota'_{k''_2,k''_3}}& 
 \cdots 
}
\]
is commutative.
We note that $\iota'_{k''_j,k''_{j+1}}$
is contiguous to $\iota_{k''_j,k''_{j+1}}$
and that $h_j\circ g_j$ is contiguous to $\iota_{k_j,k''_j}$.
We denote by $\mathcal X_{Z,C}'$ the inductive limit of the second
 horizontal line. We remark that there are no canonical map from
 $\mathcal X_{Z,C}'$ to $\mathcal X_{Z,C}$ in general.

Now we have a sequence of maps
\[
\xymatrix{
\mathcal X\ar[r]^{f_\infty}& 
\mathcal X_{Z,C}\ar[r]^{g_\infty}&
\mathcal X'\ar[r]^{h_\infty}& 
\mathcal X'_{Z,C}, 
}
\]
where we put $f_\infty:=\varinjlim f_j$,
$g_\infty:=\varinjlim g_j$ and
$h_\infty:=\varinjlim h_j$.
We prove that all maps induce isomorphisms of representable $K$-theory. 
Indeed we show that $g_\infty\circ f_\infty$ and $h_\infty\circ g_\infty$ 
induce isomorphisms of their representable $K$-theory.

We discuss only on the map $g_\infty\circ f_\infty$, 
since we can treat $h_\infty\circ g_\infty$ by the same way. 
We consider the following commutative diagram: 
\[
\xymatrix{
\abs{\U(k'_1)}\ar[r]^{\psi_{k'_1,k'_2}} &\abs{\U(k'_2)} \ar[r]^{\psi_{k'_2,k'_3}}& 
 \cdots \\
\abs{\U(1)}\ar[r]^{\psi_{1,2}} \ar[d]_{g_1\circ f_1} \ar[u]^{\psi_{1,k'_1}}
&\abs{\U(2)} \ar[r]^{\psi_{2,3}}\ar[d]_{g_2\circ f_2} \ar[u]^{\psi_{2,k'_2}}& 
 \cdots \\
\abs{\U(k'_1)}\ar[r]^{\psi'_{k'_1,k'_2}} &\abs{\U(k'_2)} \ar[r]^{\psi'_{k'_2,k'_3}}& 
 \cdots .
}
\]
The inductive limit of the first line is identified with that of
 the second line by the induced map
 $\varinjlim\psi_{j,k'_j}$. Thus
we also denote by $\mathcal X$ the inductive limit of the first line.
By Milnor exact sequence (\ref{eq:Milnor}) and its functoriality 
(see \cite[Theorem 5.8 (5)]{MR1050490} and also Proposition \ref{Phillips}), 
we have the following commutative diagram:
\[
\xymatrix{
0\ar[r]
&\limone K^{*-1}(\abs{\U(k'_j)})
\ar[r] \ar[d]_{\limone(\psi_{j,k'_j})^{*-1}}
&RK^*(\mathcal X)\ar[r]^{\varprojlim \psi_{k'_j,\infty}}
\ar[d]_{(\varinjlim\psi_{j,k'_j})^*}
&\varprojlim K^*(\abs{\U(k'_j)})\ar[r]\ar[d]^{\varprojlim (\psi_{j,k'_j})^*}
&0\\
0\ar[r]
&\limone K^{*-1}(\abs{\U(j)}))\ar[r]
&RK^*(\mathcal X)\ar[r]^{\varprojlim \psi_{j,\infty}}
&\varprojlim K^*(\abs{\U(j)})\ar[r]
&0\\
0\ar[r]
&\limone K^{*-1}(\abs{\U(k'_j)})
\ar[r]\ar[u]^{\limone(g_j\circ f_j)^{*-1}}
&RK^*(\mathcal X')\ar[r]^{\varprojlim \psi'_{k'_j,\infty}}
\ar[u]^{(g_\infty \circ f_\infty)^*}
&\varprojlim K^*(\abs{\U(k'_j)})\ar[r]\ar[u]_{\varprojlim (g_j\circ f_j)^*}
&0.
}
\]
Since $\varinjlim \psi_{j,k'_j}:\mathcal X\to \mathcal X$ is
 the identity map, $(\varinjlim \psi_{j,k'_j})^*$ is an isomorphism. 
Also $\varprojlim (\psi_{j,k'_j})^*$ is an isomorphism.
Thus so is $\limone(\psi_{j,k'_j})^{*-1}$ by the five lemma. 
Since $g_j\circ f_j$ is contiguous to $\psi_{j,k'_j}$, both
$\limone(g_j\circ f_j)^{*-1}$ and 
$\varprojlim (g_j\circ f_j)^*$ are isomorphisms, thus so is
$(g_\infty \circ f_\infty)^*$.
\end{proof}

By the definition and Milnor exact sequence~(\ref{eq:Milnor}),
$KX^*(-)$ satisfies axiom (\ref{axiom:Milnor}).

Suppose we have a proper metric space $Y$ and a coarse map
$f:X\to Y$.
We take an anti-\v{C}ech system $\{\V(k)\}_{k\in\N}$ of $Y$. 
We take an increasing sequence $\{k_j\in\N\}$ such that for each $j$, 
the covers $\U(j)$ and $\V(k_j)$ are respectively refinement of 
$\U(k_j)$ and $\V(k_{j+1})$. Then we can choose a map
$f_j:\U(j)\to \V(k_j)$ and $\phi_{k_j,k_{j+1}}:\V(k_j) \to \V(k_{j+1})$
such that $f(U)\subset f_j(U)$ for any $U\in \U(j)$ and 
the following diagram is commutative. 
\[
\xymatrix{
\abs{\U(1)}\ar[r]^{\psi_{1,2}} \ar[d]^{f_1} &\abs{\U(2)} \ar[r]^{\psi_{2,3}}\ar[d]^{f_2} & 
 \cdots\\ 
\abs{\V(k_1)}\ar[r]^{\phi_{k_1,k_2}} &\abs{\V(k_2)} \ar[r]^{\phi_{k_2,k_3}}& 
 \cdots .
}
\]
This induces a homomorphism $f^*\colon KX^*(Y)\to KX^*(X)$, which does not depend on
the choice of anti-\v{C}ech systems, the maps $f_j$ and
$\phi_{k_j,k_{j+1}}$. 
Let $g\colon X\to Y$ be another coarse map which
is close to $f$. Then we have $f^*=g^*$.
These facts can be proved by the similar arguments with the proof of 
Proposition~\ref{well-def}, so we omit the details.  

Let $Z$ be a $C$-dense uniformly discrete subset of $X$.
Then $KX^*(Z)$ coincides with the coarse
$K$-theory of $X$ defined by
Emerson-Mayer\cite{MR2225040}.
Since $Z$ and $X$ are coarsely equivalent, we
have $KX^*(Z)\cong KX^*(X)$.
Hence Emerson-Meyer's definition and ours are compatible.  

\begin{lemma}
\label{lem:product-vanishing}
 The coarse $K$-theory satisfies axiom~$(\ref{axiom:product-with-half-line})$.
\end{lemma}
\begin{proof}
Let $\{\U(k)\}_{k\in \N}$ be an anti-\v{C}ech system of $X$.
Let $\psi_k\colon \U(k)\to \U(k+1)$ denote a coarsening map. 
Set $\mathcal{V}(k) := \{U\times [n,n+k]: U\in \U(k), n\in \N\}$.
Then $\{\V(k)\}$ forms an anti-\v{C}ech system of 
$X\times \N$. For $k\in \N$ and $s\in \N\cup \{0\}$, we define a simplicial map 
$\phi_{k,s}\colon \abs{\V(k)} \to \abs{\V(k+1)}$ by 
\[
 \phi_{k,s}(U\times[n,n+k]) := \begin{cases}
		  \psi_k(U)\times[n,n+k+1] & \text{ if } n> s,\\
		  \psi_k(U)\times[n+1,n+k+2] & \text{ if } n\leq s
		 \end{cases}
\]
where $U\in \U(k)$. Since $\phi_{k,s}$ is contiguous to $\phi_{k,s+1}$,
we have a proper homotopy 
\[
 h_{k,s}(t)\colon \abs{\V(k)}\to \abs{\V(k+1)}
\] between geometric
 realization of $\phi_{k,s}$ and $\phi_{k,s+1}$ where $t\in [s,s+1]$. 
Then we define a continuous proper map 
$H_k\colon \abs{\V(k)}\times \R_{\geq 0} \to \abs{\V(k+1)}$ by 
$H_k(x,t) = h_{k,s}(t)(x)$ where $s$ is an integer satisfying 
$t\in [s,s+1]$. We remark that the restriction $H_k(-,0)$ is a
 coarsening map $\phi_{k,0}$. Thus the induced map
$\phi_{k,0}^*\colon K^*(\abs{V(k+1)})\to K^*(\abs{V(k)})$ factors through 
$K^*(\abs{V(k)}\times \R_{\geq 0}) = 0$, so 
$\varprojlim K^*(\abs{\V(k)}) =\limone K^*(\abs{\V(k)})=0$. 
Therefore $KX^*(X\times \R_{\geq 0}) = 0$.
\end{proof}

We need the following lemma to show that $KX^*(-)$ satisfies 
axiom~(\ref{axiom:Mayer-Vietoris}).
\begin{lemma}\label{pullback}
Let the following be a pullback diagram of $\sigma$-$C^\ast$-algebras:
\[
\xymatrix{
&P_k \ar[r]^{g_{1,k}} \ar[d]_{g_{2,k}}  &A_{1,k} \ar[d]^{f_{1,k}}  \\
&A_{2,k} \ar[r]^{f_{2,k}}  &B_k, 
}
\]
where we suppose that $f_{1,k}$ and $f_{2,k}$ are surjective for any $k\in \N$.
Let $\Pi_{k}:P_{k+1}\to P_k$, $\pi_{1,k}:A_{1,k+1}\to A_{1,k}$, $\pi_{2,k}:A_{2,k+1}\to A_{2,k}$ 
and $\pi_{k}:B_{k+1}\to B_k$ be $*$-homomorphisms. 
Suppose that the following diagram is commutative for every $k\in\N$
\[
\xymatrix{
P_{k+1} \ar[dr]^{\Pi_k}\ar[rr]^{g_{1,k+1}}\ar[dd]^{g_{2,k+1}}&&A_{1,k+1} \ar[dr]^{\pi_{1,k}}\ar'[d][dd]^{f_{1,k+1}}&
\\
& P_k \ar[rr]^(0.32){g_{1,k}}\ar[dd]^(0.68){g_{2,k}}& & A_{1,k} \ar[dd]^{f_1,k}
\\
A_{2,k+1} \ar'[r][rr]^{f_{2,k+1}}\ar[dr]^{\pi_{2,k}}& & B_{k+1} \ar[dr]^{\pi_k}&
\\
& A_{2,k} \ar[rr]^{f_{2,k}}& & B_k. 
}
\]
Then we have the following Mayer-Vietoris exact sequence:
\[
\xymatrix{
\ar[r]&RK_{*+1}(\varprojlim B_k) \ar[r]&RK_{*}(\varprojlim P_k) \ar[r]&
RK_*(\varprojlim A_{1,k})\oplus RK_*(\varprojlim A_{2,k}) \ar[r]&.
}
\]
\end{lemma}
\begin{proof}
We refer to the proof of \cite[Theorem 21.2.2]{MR1656031}.

By taking projective limit, we have the following commutative diagram 
\[
\xymatrix{
&P_\infty:=\varprojlim P_k 
\ar[r]^{g_{2,\infty}}_{:=\varprojlim g_{2,k}} 
\ar[d]_{g_{1,\infty}}^{:=\varprojlim g_{1,k}}  
&A_{1,\infty}:=\varprojlim A_{1,k} 
\ar[d]_{f_{1,\infty}}^{:=\varprojlim f_{1,k}}  \\
&A_{2,\infty}:=\varprojlim A_{2,k} 
\ar[r]^{f_{2,\infty}}_{:=\varprojlim f_{2,k}} 
&B_\infty:=\varprojlim B_k, 
}
\]
which is not necessarily a pull-back diagram. 
Put for each $k\in\N\cup\{\infty\}$, 
\begin{align*}
C_k:=\left\{(h_{1,k},h_{2,k})\in C_0([0,1))\otimes A_{1,k}\oplus  C_0([0,1))\otimes A_{2,k} 
\left| \ f_{1,k}(h_{1,k}(0))=f_{2,k}(h_{2,k}(0))\right.\right\}.
\end{align*}
For a $\sigma$-$C^*$-algebra $A$, we denote by $SA$ the suspension 
$C_0(0,1)\otimes A$. For each $k\in\N\cup\{\infty\}$, 
there is a canonical map $\psi_k:C_k\to SB_k$ 
defined by 
\[
[\psi_k(h_{1,k},h_{2,k})](t):=\left\{
  \begin{array}{cc}
    f_{1,k}(h_{1,k}(1-2t))   &  \text{for }t\le \frac{1}{2}  \\
    f_{2,k}(h_{2,k}(2t-1))   &  \text{for }t\ge \frac{1}{2}  \\
  \end{array}
\right..
\]
Then we have the following commutative diagram where two horizontal
 sequences are both exact, 
\[
\xymatrix{
&0  \ar[r]  
&\limone RK_{*+1}(C_k) \ar[d]^{\limone(\psi_k)_{*+1}} \ar[r]  
&RK_*(\varprojlim C_k) \ar[d]^{(\psi_\infty)_*} \ar[r] 
&\varprojlim  RK_{*}(C_k) \ar[d]^{\varprojlim (\psi_k)_*} \ar[r]
&0\\
&0 \ar[r]  
&\limone RK_{*+1}(SB_k) \ar[r]  
&RK_*(S\varprojlim B_k) \ar[r] 
&\varprojlim RK_{*}(SB_k) \ar[r]
&0.
}
\]
Since $(\psi_k)_*$ is an isomorphism for each $k$, 
so is a map $(\psi_\infty)_*$.

We have the following 
\[
\xymatrix{
&0\ar[r]& SA_{1,k+1}\oplus SA_{2,k+1}\ar[r] \ar[d]
&C_{k+1} \ar[r] \ar[d] &P_{k+1}\ar[r] \ar[d]& 0\\
&0\ar[r]& SA_{1,k}\oplus SA_{2,k}\ar[r] 
&C_{k} \ar[r]  &P_{k}\ar[r] & 0.
}
\]
Here each horizontal sequence is exact. 
(See \cite[Section 2]{MR1388297}.)
Since the left vertical map is surjective by the given condition, 
we have an exact sequence
\[
0\to SA_{1,\infty}\oplus SA_{2,\infty}
\to C_\infty \to P_\infty\to 0. 
\] 
We define 
$\kappa_\infty: SA_{1,\infty}\oplus SA_{2,\infty}\to SB_\infty$
as the restriction of $\psi_\infty$. 
Then we have the following exact sequence 
\[
\xymatrix{
\ar[r]&RK_{*+1}(P_{\infty}) \ar[r]^{\partial\qquad} 
&RK_*(SA_{1,\infty}\oplus SA_{2,\infty})
\ar[rd]_{(\kappa_\infty)_*} \ar[r] \ar[d]_\cong&
RK_*(C_{\infty}) \ar[r] \ar[d]^\cong_{(\psi_\infty)_*}& \\
&&RK_*(SA_{1,\infty})\oplus RK_*(SA_{2,\infty}) 
&RK_*(SB_\infty) &
}
\]
This gives the desired exact sequence
by $RK_{*+1}(-) \cong RK_*(S -)$. 
\end{proof}

\begin{proof}[Proof of Proposition \ref{prop:coarse_cohomology} for $KX^*(-)$]
We prove that $KX^*(-)$ satisfies axiom (\ref{axiom:Mayer-Vietoris}).
Let $X$ be a space with a proper metric $d$. 
We take a $C$-dense uniformly discrete subset $Z$ of $X$. 
We denote $\U_{Z,C}(k)$ in Proof of Claim \ref{well-def} 
by $\U(k)$ in this proof. 
It is straightforward to show the following claim.
\begin{claim}
\label{subspace}
Let $L\subset X$ be a closed subset. 
By restriction, we have an anti-\v{C}ech system 
$\{L\cap\U(k):=\left\{L\cap U \left| U\in \U(k)\right\}\right.\}_{k\in\N}$ of $L$. 
Also we consider the subcomplex $\abs{\U(k)^L}$ of $\abs{\U(k)}$ completely spanned by 
$\U(k)^L:=\left\{U\in \U(k) \left| L\cap U\neq\emptyset \right\}\right.$
for each $k\in \N$.
Then we have an injective proper simplicial map $|L\cap\U(k)|\hookrightarrow \abs{\U(k)^L}$
induced by $L\cap\U(k)\ni L\cap U\mapsto U\in\U(k)^L$.
This induces an isomorphism from $\varprojlim C_0(\abs{\U(k)^L})$ to
$\varprojlim C_0(|L\cap \U(k)|)$
as $\sigma$-$C^\ast$-algebras
and thus induces an isomorphism from 
$RK_*(\varprojlim C_0(\abs{\U(k)^L}))$ to 
$RK_*(\varprojlim C_0(|L\cap \U(k)|))$ 
\end{claim}
Note that $KX^*(L)=RK_*(\varprojlim C_0(|L\cap \U(k)|))$ in the above.

Now we consider an $\omega$-excisive decomposition $X=A\cup B$. 
Then $\abs{\U(k)}=\abs{\U(k)^A}\cup \abs{\U(k)^B}$ is an excisive decomposition 
as simplicial complexes.
Hence we have the following projective system of pull-back diagrams of $C^\ast$-algebras:
\[
\xymatrix{C_0(\abs{\U(k)}) \ar[r] \ar[d]  &C_0(\abs{\U(k)^B}) \ar[d]  \\
C_0(\abs{\U(k)^A}) \ar[r]  &C_0(\abs{\U(k)^A}\cap \abs{\U(k)^B}).}
\]
Since $\abs{\U(k)^L}\to \abs{\U(k+1)^L}$ is injective 
for any closed subspace $L\subset X$, 
Lemma \ref{pullback} implies the following exact sequence: 
\begin{align*}
\cdots \to 
RK_*(\varprojlim C_0(\abs{\U(k)^A}))\oplus RK_*(\varprojlim C_0(\abs{\U(k)^B})) 
\to RK_*(\varprojlim C_0(\abs{\U(k)}))\\
\to RK_{*-1}(\varprojlim C_0(\abs{\U(k)^A}\cap \abs{\U(k)^B}))\to \cdots .
\end{align*}
It follows from Claim \ref{subspace} that 
$KX^*(A)$, $KX^*(B)$ and $KX^*(X) $
are naturally isomorphic to 
$RK_*(\varprojlim C_0(\abs{\U(k)^A}))$, 
$RK_*(\varprojlim C_0(\abs{\U(k)^B}))$ and 
$RK_*(\varprojlim C_0(\abs{\U(k)}))$, respectively.  

Now we prove that 
$RK_*(\varprojlim C_0(\abs{\U(k)^A}\cap \abs{\U(k)^B}))$
is naturally isomorphic to $KX^*(A\cap B)$.
We have a natural injection 
$\abs{\U(k)^{A\cap B}}\hookrightarrow \abs{\U(k)^A}\cap \abs{\U(k)^B}$. 
Also we have $\abs{\U(k)^A}\cap \abs{\U(k)^B}\hookrightarrow\abs{\U(k)^{\Pen(A,2(k+1)C)\cap \Pen(B,2(k+1)C)}}$.
Since $X = A\cup B$ is an $\omega$-excisive decomposition, there
 exists $k'\in \N$ such that 
$\Pen(A,2(k+1)C)\cap \Pen(B,2(k+1)C)\subset \Pen(A\cap B,2(k'+1)C)$. 
Hence we have  
$\abs{\U(k)^{\Pen(A,2(k+1)C)\cap \Pen(B,2(k+1)C)}}\hookrightarrow\abs{\U(k)^{\Pen(A\cap B,2(k'+1)C)}}$.
Then we have $\abs{\U(k)^{\Pen(A\cap B,2(k'+1)C)}}\hookrightarrow\abs{\U((k+2k'+3)C)^{A\cap B}}$. 
By taking an increasing sequence $\{k_j\in\N\}_j$, we have the following 
commutative diagram:
\[
\xymatrix{
&{\abs{\U(k_1)^{A\cap B}}} \ar[d] \ar[r]  
&{\abs{\U(k_2)^{A\cap B}}} \ar[d] \ar[r]  
&{\abs{\U(k_3)^{A\cap B}}} \ar[d] \ar[r] 
&\cdots\\
&\ar[ur] {\abs{\U(k_1)^A}\cap \abs{\U(k_1)^B}} \ar[r]  
&\ar[ur] {\abs{\U(k_2)^A}\cap \abs{\U(k_2)^B}} \ar[r]  
&\ar[ur] {\abs{\U(k_3)^A}\cap \abs{\U(k_3)^B}} \ar[r] 
&\cdots.}
\] 
This implies that 
$\varprojlim C_0(\abs{\U(k_j)^{A\cap B}})\cong 
\varprojlim C_0(\abs{\U(k_j)^A}\cap \abs{\U(k_j)^B})$.
By combining Claim \ref{subspace}, we have that 
$RK_*(\varprojlim C_0(\abs{\U(k)^A}\cap \abs{\U(k)^B}))$
is naturally isomorphic to $KX^*(A\cap B)$.
Hence we have the desired exact sequence: 
\[
\cdots \to KX^*(A)\oplus KX^*(B) \to KX^*(X)
\to KX^{*+1}(A\cap B)\to \cdots .
\]
We can easily confirm its functoriality.


Now we show that $KX^*(-)$ satisfies axiom (\ref{axiom:character-map}). We
have a proper continuous map $X\to \abs{\U(1)}$ by using partition of
unity (see \cite[Section 3]{MR1388312}).  Then we have a $*$-homomorphism
$\varprojlim C_0(\abs{\U(k)})\to C_0(X)$.  This induces 
the character map $c:KX^*(X)\to K^*(X)$. It follows from 
Proof of the axiom (ii) that
the character maps preserve Mayer-Vietoris sequences for
$\omega$-excisive decomposition. Also the character map for
a uniformly contractible proper metric space with bounded geometry
is an isomorphism by \cite[Theorem 4.8]{MR2225040}.
We can confirm that this does not depend on the choice of partition of
unity and so on.

Finally we show that $KX^*(-)$ satisfies axiom (\ref{axiom:transgression-map}).
We consider a proper continuous map $\epsilon :X\to \abs{\U(1)}$ in the above.
Then we can give a proper coarse structure on $\abs{\U(k)}$ such that 
$\iota_{i,k}\circ\epsilon:X\to \abs{\U(k)}$ is a coarse equivalence
by using Proposition \ref{prop:pull_back_coarse_str}. 
Hence if $W$ is a corona of $X$, 
then $W$ is naturally a corona of $\abs{\U(k)}$ for each $k\in \N$. 
We have the following diagram
\[
\xymatrix{
&0 \ar[r] &C_0(\abs{\U(k+1)}) \ar[r] \ar[d]&C(\abs{\U(k+1)}\cup W) \ar[r] \ar[d]
&C(W)\ar[r] \ar[d]^{=}&0\\
&0 \ar[r] &C_0(\abs{\U(k)}) \ar[r] &C(\abs{\U(k)}\cup W) \ar[r]&C(W)\ar[r] &0, 
}
\]
where we can assume that left vertical map is
 surjective without loss of generality. Hence we have 
\[
\xymatrix{
&0 \ar[r] &\varprojlim C_0(\abs{\U(k)}) \ar[r] &\varprojlim C(\abs{\U(k)}\cup W) \ar[r]&C(W)\ar[r] &0.}
\]
The map $\epsilon$ induces the following: 
\[
\xymatrix{
&0 \ar[r] &\varprojlim C_0(\abs{\U(k)}) \ar[r]\ar[d] 
&\varprojlim C(\abs{\U(k)}\cup W) \ar[r]\ar[d]&C(W)\ar[r] \ar[d]^=&0\\
&0 \ar[r] &C_0(X) \ar[r] &C(X\cup W) \ar[r]&C(W)\ar[r] &0.}
\]

Since the inclusion $\C\to C(W)$ factors through 
$\varprojlim C(\abs{\U(k)}\cup W) \to C(W)$,  we have
\[
\xymatrix{
&\tilde{K}^{*-1}(W)\ar[r]^{T_W} \ar[dr]^{\partial}&KX^*(X)\ar[d]^c\\
&&K^*(X).}
\]
From the construction, it is
easy to see that the transgression map is natural in the sense of
 axiom~(\ref{axiom:transgression-map}). 
\end{proof}

\subsection{The coarse co-assembly map}
\label{sec:coarse-co-assembly}
Let $X$ be a proper metric space. We denote by
$B(\mathcal{H})$ the $C^*$-algebra of bounded linear operators
on a separable infinite dimensional Hilbert space $\mathcal{H}$. We also
denote by $\K$ the $C^*$-algebra of compact operators on $\mathcal{H}$.
\begin{definition}[\cite{MR2225040}]
 \label{def:stable-Higson-corona}
We let $\barHig$ be the $C^*$-algebra of bounded continuous
 $B(\mathcal{H})$-valued Higson functions on $X$ such that 
$f(x) - f(y) \in \K$ for all $x,y\in X$. The quotient 
$\rsHig(X):= \barHig(X)/C_0(X,\K)$ is called 
the reduced stable Higson corona of $X$. 
\end{definition}
See \cite[Definition 4.3]{MR2225040} for the unreduced stable
Higson corona.
\begin{proposition}[\cite{MR2225040}]
 \label{prop:funtority-stable-Higson}
The assignment $X\mapsto \rsHig(X)$ is a contravariant functor from the
 coarse category to the category of $C^*$-algebras.
\end{proposition}
Let $\{\U_n\}$ be an anti-\v{C}ech system of $X$. 
We fix coarsening maps $\abs{\U_n}\to \abs{\U_{n+1}}$ and put
$\mathcal{X}:= \varinjlim \abs{\U_n}$. Then we have 
canonical maps $\Psi_n\colon \abs{U_n}\to \mathcal{X}$.
We put
\begin{align*}
 C_0(\mathcal{X},\K)&:= 
\{f\colon \mathcal{X}\to \K: f\circ \Psi_n\in C_0(\abs{U_n},\K)
\textrm{ for all }n\in \N\};\\
 \barHig(\mathcal{X})&:= 
\{f\colon \mathcal{X}\to B(\mathcal{H}): f\circ \Psi_n \in \barHig(\abs{U_n})
\textrm{ for all }n\in \N\}.
\end{align*}
Both of $ C_0(\mathcal{X},\K)$ and $\barHig(\mathcal{X})$ are 
$\sigma$-$C^*$-algebras. We have
\[
  C_0(\mathcal{X},\K)= \varprojlim C_0(\abs{U_n},\K),\; 
  \barHig(\mathcal{X}) = \varprojlim \barHig(\abs{U_n}).
\]
Since coarsening maps $X\to \abs{U_n}$ and
$\abs{U_{n}}\to \abs{U_{n+1}}$ are coarse equivalences,
Proposition~\ref{prop:funtority-stable-Higson} implies that the
projective limit
\[
 \rsHig(\mathcal{X}) := \varprojlim \rsHig(\abs{U_n})
\]
is again a $C^*$-algebra, which is isomorphic to $\rsHig(X)$.
The following sequences of $\sigma$-$C^*$-algebras is exact 
(\cite[Lemma 3.12]{MR2225040}).
\begin{align}
\label{exact:sHig}
 0 \to C_0(\mathcal{X},\K) \to \barHig(\mathcal{X}) \to
 \rsHig(\mathcal{X}) \to 0.
\end{align}

\begin{definition}[\cite{MR2225040}]
\label{def:coarse-co-assembly-map}
 Let $X$ be a proper metric space. The coarse co-assembly map for $X$ is
 the map
\[
 \mu^* \colon K_{*+1}(\rsHig(X)) \to KX^*(X)
\]
that is obtained from the connecting map of the exact 
sequence~(\ref{exact:sHig}).
\end{definition}

\begin{proposition}[Emerson-Meyer, Willett]
\label{prop:K-theory-of-sHig-is-cohomology}
 The $K$-theory of the reduced stable Higson corona is a generalized
 coarse cohomology theory.
\end{proposition}
\begin{proof}
 The axiom~(\ref{axiom:product-with-half-line}) follows from
 \cite[Theorem 5.2.]{MR2225040}. The axiom~(\ref{axiom:Mayer-Vietoris}) is
 proved in \cite[Proposition 4.3.6]{WillettThesis}.
\end{proof}


The Mayer-Vietoris exact sequences for both of $K_*(\rsHig(-))$ and
$KX^*(-)$ come from the general notion of the  Mayer-Vietoris exact sequence
associated to a pull-back diagram of $C^*$-algebras. 
(See \cite[Theorem 21.2.2]{MR1656031}.)
Therefore, the connecting maps in both of these exact sequences and
coarse co-assembly maps are naturally commutative. That is, for an
$\omega$-excisive decomposition $X=A\cup B$ of a proper metric space $X$,
we have the following commutative diagram, 
\begin{align*}
 \xymatrix{
\ar[r] &K_{p+1}(\rsHig(X)) \ar[r]\ar[d]  
       &K_{p+1}(\rsHig(A))\oplus K_{p+1}(\rsHig(B)) \ar[r]\ar[d]  
       &K_{p+1}(\rsHig(A\cap B))\ar[r]\ar[d]  &\\
\ar[r] &KX^{p}(X) \ar[r]
       &KX^{p}(A)\oplus KX^{p}(B) \ar[r]
       &KX^{p}(A\cap B)\ar[r] &
}
\end{align*}
where both of horizontal sequences are exact and vertical maps are
coarse co-assembly maps.

\section{Coarse cohomology of hyperbolic metric spaces}
\label{sec:coarse-cohom-hyperbolic-metric-spaces}
In this section, we summarize the result of \cite{ROEHYPERBOLIC} and
\cite{MR1388312} from the view point of the coarse cohomology theories.
Let $M^*$ be the $K$-theory or the Alexander-Spanier
cohomology with compact supports and let $MX^*$ be its
coarsening.

\subsection{The transgression map of the open cone}
\label{sec:transgr-map-open}
Let $Y$ be a compact subset of the unit sphere in a
separable Hilbert space $H$.
The open cone on $Y$, denoted $\mathcal{O}Y$, is the set of all non-negative
multiples of points in $Y$. The closed cone 
$\mathcal{C}Y = \{tx \in H: t\in [0,1], x\in Y\}$ is a compactification
of $\mathcal{O}Y$ and $Y$ is a corona of it. 
By axiom~(\ref{axiom:transgression-map}), there is a commutative
diagram. (See also \cite[Example 5.28]{MR1147350}.)
\begin{align}
\label{diagram:transgression}
\xymatrix{
& MX^{q}(\mathcal{O}Y) \ar[dd]_c\\
\tilde{M}^{q-1}(Y) \ar[ur]^{T_Y} \ar[dr]^{\partial}&\\
& M^{q}(\mathcal{O}Y)
}
\end{align} 
Here $T_Y$ is a transgression map and $\partial$ is the boundary map in
the long exact cohomology sequence for $Y\subset \mathcal{C}Y$.
\begin{lemma}
\label{lem:transg_map_for_cone}
 The character map 
 $c\colon MX^{q}(\mathcal{O}Y) \to M^q(\mathcal{O}Y)$ and 
 the transgression map
 $T_{Y} \colon \tilde{M}^{q-1}(Y) \to  MX^q(\mathcal{O}Y)$ are
 isomorphisms.
\end{lemma}
\begin{proof}
First, we consider a
 cohomology long exact sequence for $Y\subset \mathcal{C}Y$. 
Since $\mathcal{C}Y$ is homotopic to one point, the long exact sequence
 splits and we obtain 
\[
 0 \to M^{q-1}(\mathcal{C}Y) \to M^{q-1}(Y) \xrightarrow{\partial } 
 M^{q}(\mathcal{O}Y) \to 0.
\]
Hence $\partial \colon \tilde{M}^{q-1}(Y)\to M^q(\mathcal{O}Y)$ is an isomorphism.

Next, let $\{\mathcal{U}_i\}$ be an anti-\v{C}ech system of $\mathcal{O}Y$
constructed in the proof of \cite[Proposition 4.3]{MR1388312} (see also
\cite[Appendix B]{relhypgrp}). Then it is shown that:
\begin{itemize}
 \item Each $\abs{\U_i}$ is equipped with a
       proper coarse structure which is coarsely equivalent to
       $\mathcal{O}Y$, so $Y$ is also a corona of $\abs{\U_i}$. 
       Thus we have a coarse compactification 
       $\overline{\abs{\U_i}}:= \abs{\U_i}\cup Y$.
 \item The coarsening map $\abs{\U_i} \to \abs{\U_{i+1}}$
       covers the identity on $Y$.
 \item The extended map 
       $\overline{\abs{\U_i}} \to \overline{\abs{\U_{i+1}}}$ is
       nullhomotopic.
\end{itemize}
By the argument similar to the proof of 
\cite[Proposition 4.3]{MR1388312},
 we can show that the boundary map $\partial$ gives an
 isomorphism between $\tilde{M}^{q-1}(Y)$ and
 $\Im[M^q(\abs{\U_{i+1}}) \to M^q(\abs{\U_i})]$. This
 implies $\limone M^q(\abs{\U_i}) = 0$ and 
 $\tilde{M}^{q-1}(Y) \cong \varprojlim M^q(\abs{\U_i})$. 
Thus it follows from axiom~(\ref{axiom:Milnor}) that 
the character map $c\colon MX^q(\mathcal{O}Y)\to M^q(\mathcal{O}Y)$ is
 an isomorphism. Now the diagram (\ref{diagram:transgression}) 
 shows that the transgression map $T_Y$ is an isomorphism.
\end{proof}

\subsection{Hyperbolic spaces}
\label{sec:hyperbolic-spaces}
Let $X$ be a proper geodesic space which
is hyperbolic in the sense of
Gromov. Roe~\cite{ROEHYPERBOLIC} showed
that the Gromov boundary of $X$, denoted by $\partial X$, is a
corona of $X$.
Higson-Roe~\cite{MR1388312} constructed a coarse map
$\mathcal{O}(\partial X) \to X$
and showed that it is a coarse homotopy equivalence.  Thus by coarse
homotopy invariance, we have 
$MX^*(X) \cong MX^*(\mathcal{O}(\partial X))$. 
For details, see \cite[Section 8]{MR1388312} and 
\cite[Section4.7]{WillettThesis}. By the same reason, we have 
$K_*(\rsHig(X)) \cong K_*(\rsHig(\mathcal{O}(\partial X)))$. 
Willett \cite[Section 4.5]{WillettThesis} showed that the coarse
co-assembly map for the open cone $\mathcal{O}(\partial X)$ is an
isomorphism. Therefore we have the following.


\begin{proposition}
\label{prop:coarse_co-assemby_for_hyp}
 Let $X$ be a proper geodesic space which is
 hyperbolic in the sense of Gromov. Then the coarse co-assembly map 
$\mu^*\colon K_{*+1}(\rsHig(X)) \to KX^*(X)$ is an
 isomorphism. 
\end{proposition}
It is easy to see that the coarse map $\mathcal{O}(\partial
X)\rightarrow X$ covers the 
identity on $\partial X$. Therefore, 
by Lemma~\ref{lem:transg_map_for_cone}, axiom
(\ref{axiom:transgression-map})
and coarse homotopy invariance, we have the following.
\begin{corollary}
\label{cor:transgression}
Let $X$ be a non-compact proper geodesic space which is hyperbolic in the
 sense of 
 Gromov. The transgression maps 
\begin{align*}
 T_{\partial X}\colon& KX_*(X)\to \tilde{K}_{*-1}(\partial X);\\
 T_{\partial X}\colon& \tilde{K}^{*-1}(\partial X)\to KX^*(X);\\
 T_{\partial X}\colon& \tilde{H}^{*-1}(\partial X)\to HX^*(X).
\end{align*}
are isomorphisms.
\end{corollary}

\section{Relatively hyperbolic groups}
\label{sec:relat-hyperb-groups}
Let $G$ be a finitely generated group with a finite family of infinite
subgroups $\famP = \{P_1,\dots P_k\}$.  Groves and Manning
\cite{MR2448064} introduced an {\itshape augmented space} on which $G$ acts
properly discontinuously by isometries. The augmented space characterize
hyperbolicity of $G$ relative to $\famP$.
We review the construction and show that there exists a weak coarsening
of the augmented space for cohomology theories.
\begin{remark}
 Suppose that $G$ is hyperbolic relative to $\famP$. If $P=\emptyset$, then
 $G$ is hyperbolic and thus Theorem~\ref{th:coarse_co-assembly-map} and
 Theorem \ref{main_theorem} follow from
 Proposition~\ref{prop:coarse_co-assemby_for_hyp} and
 Corollary~\ref{cor:transgression}. 
If $G\in \famP$ then $\famP=\{G\}$,
 thus Theorem~\ref{th:coarse_co-assembly-map} and
 Theorem \ref{main_theorem} are trivial. It is well known that all
 elements are of infinite index of $G$ if $G\notin \famP$.
\end{remark}
From now on, we assume that $\famP$ is not empty and all elements of
$\famP$ are of infinite index in $G$. 
\subsection{The augmented space}
\label{sec:augmented-space}
\begin{definition}
\label{def:Horoball}
Let $(P,d)$ be a proper metric space.
{\itshape The combinatorial horoball} based on $P$, denoted by
 $\Horo(P)$, is the graph defined as follows:
\begin{enumerate}
 \item $\Horo(P)^{(0)} = P \times (\N\cup \{0\})$. 
 \item $\Horo(P)^{(1)}$ contains the following two type of edges:
       \begin{enumerate}[(i)]
	\item For each $l\in \N\cup\{0\}$ and $p,q\in P$, 
	      if $0< d(p,q)\leq 2^{l}$
	      then there is a 
	      {\itshape horizontal edge} connecting $(p,l)$ and $(q,l)$.
	\item For each $l\in \N\cup\{0\}$ and $p\in P$, there is a 
	      {\itshape vertical edge} connecting $(p,l)$ and $(p,l+1)$.
       \end{enumerate}
\end{enumerate}
We endow $\Horo(P)$ with the graph metric. 
\end{definition}

When $P$ is a discrete proper metric space, 
$\Horo(P)$ is a proper geodesic space
 which is hyperbolic in the sense of Gromov. 
(See \cite[Theorem 3.8]{MR2448064}). 
It is easy to see that $\Horo(P)$ is coarsely flasque.
The following is used in Section~\ref{sec:bound-relat-hyperb}.
\begin{lemma}
\label{lem:parabolic-pt-of-horoball} Let $P$ be a proper metric space. We
suppose that $P$ is discrete. Then the Gromov compactification of the
combinatorial horoball $\Horo(P)$ is a one-point compactification of
$P$. Thus the Gromov boundary of $\Horo(P)$ consists of one point,
 called the parabolic point of $\Horo(P)$.
\end{lemma}
\begin{proof}
 See Lemma 3.11. in \cite{MR2448064}.
\end{proof}


Let $G$ be a finitely generated group with a finite family of
infinite subgroups 
$\famP = \{P_1,\dots, P_k\}$. 
We take a finite generating set $\mathcal{S}$ for $G$. 
We assume that
$\mathcal{S}$ is symmetrized, so that $\mathcal{S} = \mathcal{S}^{-1}$.
We endow $G$ with the left-invariant word metric
$d_{\mathcal{S}}$ with respect to $\mathcal{S}$. 

\begin{definition}
\label{def:order-of-cosets}
 Let $G$ and $\famP$ be as above. An order of the cosets of $(G,\famP)$
 is a sequence $\{g_n\}_{n\in \N}$ such that
 $g_i=e$ for $i\in \{1,\dots,k\}$, and 
for each $r \in
\{1,\dots,k\}$, the map 
$\N\to G/P_r: a \mapsto g_{ak+r}P_r$ is bijective. 
Thus the set of all cosets $\bigsqcup_{r=1}^k G/P_r$ is indexed by the map
$\N\ni i \mapsto g_iP_{(i)}$. Here $(i)$ denotes the remainder of $i$ divided by $k$.

\end{definition}
We fix an order $\{g_n\}_{n\in\N}$ of the cosets of $(G,\famP)$.
Each coset $g_iP_{(i)}$ has a proper metric $d_i$ which is the
restriction of $d_{\mathcal{S}}$. 
Let $\Gamma$ be the Cayley graph of $(G,\mathcal{S})$.
There exists a natural embedding 
$\psi_i\colon \Horo(g_iP_{(i)};\{0\}) \hookrightarrow \Gamma$ such that
$\psi_i(x,0) = x$ for all $x\in g_iP_{(i)}$. 
\begin{definition}
\label{notation:1}
{\itshape The augmented space } $\Xaug$ is obtained by pasting
 $\Horo(g_iP_{(i)})$
 to $\Gamma$ by $\psi_i$ for all $i\in \N$. 
Thus we can write it as follows:
\begin{align*}
 \Xaug &:= \Gamma\cup \bigcup_{i\in \N} \Horo(g_iP_{(i)}).
\end{align*}
\end{definition}

\begin{remark}
\label{rem:vertices}
 The vertex set of $\Xaug$ can naturally identified with 
the disjoint union of $G$ and the set of 
3-tuple $(i,p,l)$, where $i\in \N$, $p\in g_iP_{(i)}$, and $l\in \N$.
We sometimes denote $g\in g_iP_{(i)}$ by $(i,g,0)$ for simplicity.
\end{remark}

\begin{definition}
A group $G$ is hyperbolic relative to $\famP$ if the augmented space
$\Xaug$ is hyperbolic in the sense of Gromov.
\end{definition}
Groves and Manning \cite{MR2448064} showed that the above definition is
equivalent to the original one by Gromov.

\subsection{Weak coarsening of relatively hyperbolic groups}
\label{subsec:weak-coars-relat}
In this section, we construct a topological
counterpart of the augmented
space, which is the key to the proof of
Theorem~\ref{th:coarse_co-assembly-map} and Theorem~\ref{main_theorem}. 
Let $G$ be a finitely generated group which is hyperbolic relative to 
$\famP =\{P_1,\dots,P_k\}$.
Here we assume that for $r\in \{1,\dots,k\}$, each $P_r$ admits a finite
 $P_r$-simplicial complex $\EP_r$ which is a universal space for proper
 actions.
By \cite[Appendix A]{relhypgrp}, there exists a 
 finite $G$-simplicial complex $\EG$ which is a universal space
 for proper actions such that all $\EP_{r}$ are
 embedded in $\EG$.  We can assume that $G$ is naturally embedded in
 the set of vertices of $\EG$ and $g_iP_{(i)}$ is embedded in
 $g_i\EP_{(i)}$.  

We define an embedding 
$\eta_i \colon g_i\EP_{(i)}\times \{0\} \hookrightarrow \EG$ 
as $\eta_i(x,0) = x$.
We define a space  $\EX$ in LCSH
by pasting
$g_i\EP_{(i)}\times [0,\infty)$ to $\EG$ by $\eta_i$ for all 
$i\in \N$. Thus we can write it as follows:
\begin{align*}
 \EX &:= \EG\cup \bigcup_{i\in \N} (g_i\EP_{(i)}\times [0,\infty)).
\end{align*}


In the rest of this section, we show that $\EX$ is a weak coarsening of
$\Xaug$, that is, $MX^*(\Xaug) \cong M^*(\EX)$. 
Here $M^*$ is the $K$-theory $K^*$ or the Alexander-Spanier cohomology
with compact support $H_c^*$.

We can regard $\EX$ as a metric simplicial complex in the sense of 
\cite[Definition 3.1]{MR1388312}. However, the bounded coarse structure
associated to this metric is not coarsely equivalent to
$\Xaug$. Therefore we equip $\EX$ with a pull-back coarse structure as
follows.

Let $\Xaug^{(0)}$ denote the $0$-skeletons of $\Xaug$. 
Since $G$ and $P_r$ for $r=1,\dots,k$ are embedded respectively into
$\EG$ and $\EP_r$, there is a natural embedding 
$\iota \colon \Xaug^{(0)} \hookrightarrow \EX$.
We define a left inverse $\varphi$ of $\iota$ as follows. 
We take a finite subcomplex
$\Delta\subset \EG$ containing a fundamental domain of $\EG$.  We may assume
that $\Delta_r:=\Delta\cap \EP_r$ contains a fundamental domain of $\EP_r$ for
$r=1,\dots,k$ without loss of generality. Then we can write $\EX$ as
follows.
\[
 \EX = \bigcup_{g\in G}g\Delta\cup 
 \bigcup_{i\in \N}\bigcup_{h\in P_{(i)}}g_ih\Delta_{(i)}\times (0,\infty).
\]
For every $x\in \EG$, we choose $g_x \in G$ such that $x\in
g_x\Delta$ and put $\varphi(x) := g_x\in \Gamma$. 
For $(x,t)\in g_ih\Delta_{(i)}\times (0,\infty)$, we put 
$\varphi(x,t):= (i,g_ih,[t]) \in \Horo(g_iP_{(i)})$
where $[t]$ denotes the integral part of $t$.
 We equip $\EX$ with a pullback coarse structure by $\varphi$. It is
easy to see that $\iota$ and $\varphi$ satisfy the conditions in
Proposition~\ref{prop:pull_back_coarse_str} and
Proposition~\ref{prop:pullback-proper}. Therefore $\EX$ is a proper coarse
space which is coarsely equivalent to
$\Xaug$. By the construction, $\EG$ and $\EP_i$ with the restricted coarse
structure are respectively coarsely equivalent to $G$ and $P_i$. Since
$G$ is finitely generated, $\EG$ has bounded geometry in the sense of
\cite[Definition 3.9]{MR2007488} and is uniformly contractible in the sense of
\cite[Definition 5.24]{MR2007488}, and so does $\EP_i$.

In Section 2.3 and Section 3.1 of \cite{relhypgrp}, the followings are defined.
\begin{enumerate}
 \item An anti-\v{C}ech system $\{\U_n\}_{n}$ of $\Xaug$.
 \item Coarsening maps $\alpha_n\colon \U_n \to \U_{n+1}$.
 \item Subsets $\mathcal{X}_n,\, \mathcal{Y}_n, \mathcal{Z}_n$ of
       $\U_n$.
 \item An anti-\v{C}ech system $\{E\U_n\}_{n}$ of $\EX$ in the sense of
       \cite[Definition 5.36]{MR2007488}.
 \item Simplicial maps $\phi_n\colon E\U_n\to \U_{n+1}$.
\end{enumerate}
A partition of
unity defines a continuous map $\psi\colon \EX\to E\U_1$. 
For $n\geq 3$, set
$F_n := \alpha_{n-1}\circ \dots \circ \alpha_{2} \circ \phi_1 \circ
\psi\colon \EX \to \abs{\U_n}$. 
We remark that the image of the restriction of $F_n$ to 
$\EG$ lies on $\abs{\mathcal{X}_n}$.
Then we have the following commutative diagram.
\begin{align*}
\xymatrix{
 \ar[r]& M^p(\abs{\U_{n+1}}) \ar[r]\ar[d]& M^p(\abs{\mathcal{X}_{n+1}}) 
 \oplus M^p(\abs{\mathcal{Y}_{n+1}}) \ar[r]\ar[d]&
 M^p(\abs{\mathcal{Z}_{n+1}}) \ar[r]\ar[d]& \\
 \ar[r]& M^p(\abs{\U_n}) \ar[r]\ar[d]& M^p(\abs{\mathcal{X}_n}) 
 \oplus M^p(\abs{\mathcal{Y}_n}) \ar[r]\ar[d]&
 M^p(\abs{\mathcal{Z}_n}) \ar[r]\ar[d]& \\
\ar[r] & M^p(\EX) \ar[r]& M^p(\EG) \ar[r]&
 M^p(\bigsqcup_{i\in \N}g_i\EP_{(i)}) \ar[r]&. 
}
\end{align*}
Here a map 
$M^p(\abs{\mathcal{X}_n}) \oplus M^p(\abs{\mathcal{Y}_n}) \to 
M^p(\EG)$ is given by $(a,b)\mapsto F_n^*(a)$.
Since $\EG$ and $\EP_i$ are of bounded geometry, 
 uniformly contractible coarse spaces, by the same
 way as in the proof of \cite[Proposition~3.8]{MR1388312}, taking
 subsequence if necessary, we can show that 
$\Im[M^*(\abs{\mathcal{X}_{n+1}})
\to M^*(\abs{\mathcal{X}_n})] \cong M^*(\EG)$ and 
$\Im[M^*(\abs{\mathcal{Z}_{n+1}})\to M^*(\abs{\mathcal{Z}_n})]\cong 
M^*(\bigsqcup_{i\in \N}g_i\EP_{(i)})$ 
for all $n\geq 1$.
By the same argument as in the proof of \cite[Lemma~2.7]{relhypgrp}, 
we can show that 
$\Im[M^*(\abs{\mathcal{Y}_{n+1}})\to M^*(\abs{\mathcal{Y}_n})]=
 0$. Thus by diagram chasing, we have 
$\Im[M^*(\abs{\U_{n+1}})\to M^*(\abs{\U_n})]\cong M^*(\EX)$ for all
 $n\geq 1$. 
Therefore we have $\limone M^*(\abs{\U_n}) = 0$
and $\varprojlim M^*(\abs{\U_n}) \cong M^*(\EX)$.
By axiom~(\ref{axiom:Milnor}), we have the 
following conclusion.
\begin{proposition}
 \label{prop:weak-coarsening}
The space $\EX$ is a weak
 coarsening of $\Xaug$, that is, $MX^*(\Xaug) \cong M^*(\EX)$.
\end{proposition}

We use the following notations introduced in \cite{relhypgrp}
\begin{align*}
 X_n &:= \Gamma \cup \bigcup_{ i > n}\Horo(g_iP_{(i)});\\
 X_\infty &:= \bigcap_{n> 0} X_n;\\
 EX_n &:= 
   \EG \cup \bigcup_{i > n}
     (g_i\EP_{(i)}\times [0,\infty));\\
 EX_\infty &:= \bigcap_{n> 0} EX_n.
\end{align*}
We remark that $X_0=\Xaug$, $X_\infty = \Gamma$, $EX_0 = EX$ and 
$EX_\infty = \EG$. We note that the definition of $X_n$ is
slightly different from the one in \cite{relhypgrp}, that is, the index
is shifted by one.
By the Mayer-Vietoris argument and
Proposition~\ref{prop:weak-coarsening}, we have the following
\begin{proposition}
\label{prop:X_n=EX_n}
 The following is commutative for all $n\in \N$
\begin{align*}
 \xymatrix{
MX^*(X_n)\ar[d] \ar[r]^\cong & M^*(EX_n) \ar[d]\\
MX^*(X_{n+1}) \ar[r]^\cong & M^*(EX_{n+1}).
}
\end{align*}
\end{proposition}
By the continuity of $M^*$, we have $\varinjlim M^*(EX_n)\cong
M^*(\EG)$. Since $\EG$ is a finite model, we have $MX^*(G)\cong
M^*(\EG)$. Hence Proposition~\ref{prop:X_n=EX_n} implies the following.
\begin{corollary}
\label{cor:limX_n=G}
 We have an isomorphism $\varinjlim MX^*(X_n)\cong MX^*(G)$.
\end{corollary}
\subsection{Coarse assembly map and its dual}
\label{subsec:coarse-assembly-map}
In this section, we give a proof of
Theorem~\ref{th:coarse_co-assembly-map}.
 The first statement is proved in \cite{relhypgrp}. The second statement
 is proved by a similar way. 
We suppose that $\famP$ satisfies the
 condition in Theorem~\ref{th:coarse_co-assembly-map}, that is, the
 coarse co-assembly map is an isomorphism for all $P\in \famP$.

By Proposition~\ref{prop:coarse_co-assemby_for_hyp}, the coarse
 co-assembly map $\mu^*\colon K_{*+1}(\rsHig(X_0)) \to KX^*(X_0)$
 is an isomorphism. Since $X_{n} = X_{n+1}\cup \Horo(g_{n+1}P_{(n+1)})$
 is an $\omega$-excisive 
 decomposition, by using the
Mayer-Vietoris sequences, we can show that for all $n\in \N$, the coarse
 co-assembly map $\mu^*\colon K_{*+1}(\rsHig(X_n)) \to KX^*(X_n)$
 is an isomorphism. 
Finally, by the continuity of the $K$-theory 
and Corollary~\ref{cor:limX_n=G} we have 
\begin{align*}
\xymatrix{
\varinjlim K_{*+1}(\rsHig(X_n)) \ar[d]^\cong\ar[r]^\cong& 
\varinjlim KX^*(X_n)\ar[d]^\cong\\
K_{*+1}(\rsHig(G)) \ar[r]& KX^*(G).
}
\end{align*}

The following is a somewhat converse statement of Theorem 1.1. However,
we assume nothing on universal spaces for proper actions.

\begin{proposition}\label{prop:coarse_assmap-for-parabolic-subgroup}
Let $G$ be a group which is hyperbolic relative to $\famP$.
\begin{enumerate}
\item \label{item:converse-assembly}
If $\mu_*:KX_*(G)\cong K_*(C^*(G))$, then
$\mu_*: KX_*(P)\cong K_*(C^*(P))$ for every $P\in\famP$.
\item \label{item:converse-co-assembly}
If $\mu^*:K_{*-1}(\rsHig(G))\cong KX^*(G)$, then
$\mu^*: K_{*-1}(\rsHig(P))\cong KX^*(P)$ for every $P\in\famP$.
\end{enumerate}
\end{proposition}

\begin{proof}
We fix $r\in \{1,\dots,k\}$. Set 
$A := \Gamma \cup \bigcup_{i\neq r}\Horo(g_iP_i)$ 
 and $B:= \Gamma \cup \Horo(g_rP_r)$, Then $\Xaug=A\cup B$ and 
$B= \Gamma\cup \Horo(g_rP_r)$ are $\omega$-excisive decompositions. 
By the Mayer-Vietoris arguments for $A\cup B$, we have 
$\mu_*\colon KX_*(B)\to K_*(C^*(B))$ and 
$\mu^*\colon K_*(\rsHig(B))\to KX^*(B)$ are both
 isomorphisms. 
By the Mayer-Vietoris arguments for 
$B= \Gamma\cup \Horo(g_rP_r)$, we have 
$\mu_*\colon KX_*(\Gamma\cap \Horo(g_rP_r))\to K_*(C^*(\Gamma\cap
 \Horo(g_rP_r)))$ and 
$\mu^*\colon K_*(\rsHig(\Gamma\cap \Horo(g_rP_r))\to 
KX^*(\Gamma\cap \Horo(g_rP_r))$ are both
 isomorphisms. Here we use the fact that $\Horo(g_rP_r)$ is
 coarsely flasque. Since $\Gamma\cap \Horo(g_rP_r)$ is
 coarsely equivalent to $P_r$, we have the conclusion.
\end{proof}

\section{Corona of relatively hyperbolic groups}
\label{sec:bound-relat-hyperb}
In this section, we construct a corona of a relatively hyperbolic
group. Here we sketch the construction. Let $(G,\famP)$ be a
relatively hyperbolic group. We fix a generating set $\mathcal{S}$
of $G$ and an order $\{g_n\}_{n\in\N}$ of the
cosets of $(G,\famP)$ in the sense of Definition~\ref{def:order-of-cosets}. 
The Bowditch boundary $\partial \Xaug$ contains no
information on a maximal parabolic subgroup $P$ because all orbits by $P$
go to a single parabolic point $s\in \partial \Xaug$.
We remove the parabolic point $s$ and equip $\partial \Xaug
\setminus \{s\}$ with a coarse structure which is coarsely
equivalent to $P$.
Let $(W,\zeta)$ be a corona of $P$. Then $(W,\zeta)$ is also a corona of 
$\partial \Xaug\setminus \{s\}$. Thus we obtain a 
blown-up
$\partial \Xaug\setminus \{s\} \cup W$. Repeating this procedure to all
parabolic points, we obtain a corona $\partial X_\infty$ of $G$.

\subsection{A coarse structure on the complement of a parabolic point}
Let $G$ be a group which is hyperbolic relative to $\famP$.
For  $p,x,y\in \Xaug$,
we denote by $(x|y)_p$ the Gromov product 
\[
 (x|y)_{p} := \frac{1}{2}(d(x,p)+d(y,p) - d(x,y)).
\]
We denote by $[x,y]$ a geodesic connecting $x$ and $y$.
Since $\Xaug$ is hyperbolic in the sense of Gromov, there exists
$\delta_0>0$ such that every geodesic
triangle is $\delta_0$-thin, that is, for any $x,y,z \in \Xaug$, 
and for any $u\in [x,y]$ and $v\in [x,z]$, 
if $d(x,u)=d(x,v)\leq (y|z)_x$, then
$d(u,v)\leq \delta_0$. For details, see \cite[Chapter 2]{MR1086648}.

Two geodesic rays in $\Xaug$ are said to be equivalent if the Hausdorff
distance of their images is finite. For a geodesic ray $l\colon
[0,\infty)\to \Xaug$, we denote by $[l]$ the equivalent class of $l$.
We also write $l(\infty)=[l]$. 
The Gromov boundary of $\Xaug$, denoted by $\partial \Xaug$, consists of
equivalent classes of geodesic rays. 
It carries a natural topology and $\barXaug
:= \Xaug\cup \partial\Xaug$ is a compactification of $\Xaug$.
The Gromov product is extended on $\barXaug$ as follows. 
For $u,v\in \barXaug$ and $p \in\Xaug$, we put
\[
 (u|v)_{p} := \sup \liminf_{i,j\to \infty} (x_i|y_j)_p
\]
where the supremum is taken over all sequences $(x_i)_{i\geq1}$ and
$(y_i)_{i\geq1}$ tending to $u$ and $v$, respectively. 
For details, see \cite[Chapter~7]{MR1086648}.
Let $l_0, l_1\colon [0,\infty)\rightarrow
\Xaug$ be geodesic rays such that 
$p:=l_0(0) = l_1(0)$. Then it is easy to see that $(l_0(s)|l_1(t))_p$ is
non-decreasing for all $s,t\geq 0$, thus we have 
$([l_0]|[l_1])_p\geq (l_0(s)|l_1(t))_p$ for all $s,t\geq0$.
The following is known.
\begin{lemma}
\label{lem:approx}
In the above
 setting, there exists $t_0$ such that for all $s,t\geq t_0$, we have 
 $(l_0(s)|l_1(t))_p \geq ([l_0]|[l_1])_p - 3\delta_0.$
\end{lemma}
\begin{proof}
 The lemma follows immediately from \cite[Remark 7.2.8]{MR1086648}.
\end{proof}
The augmented space have the following {\itshape tautness}.

\begin{lemma}
 \label{lem:taut}
The augmented space $\Xaug$ is taut, in fact, 
for any vertex $x\in \Xaug$, there exists a bi-infinite geodesic 
$l\colon (-\infty,\infty)\to \Xaug$ such that $x$ lies on $l$.
\end{lemma}

\begin{proof}
Take any vertex $(i,g,n)\in \Xaug$. (See Remark~\ref{rem:vertices}, we
 often use this notation.)
We choose $j\in \N$ such that 
$\Horo(g_iP_{(i)})\cap \Horo(g_{j}P_{(j)}) = \emptyset$. 
Then we choose a shortest geodesic
$\gamma:[0,a]\to \Xaug$ connecting $\Horo(g_iP_{(i)})$ and 
$\Horo(g_{j}P_{(j)})$. We remark that its end points $p:=l(0)$ and $q:=l(a)$ 
lie respectively on $g_iP_{(i)}$ and $g_{j}P_{(j)}$. 
We take the vertical ray $\gamma_{-}\colon [0,\infty)\to \Xaug$ from $p$ to 
the parabolic point $s_i$ of $\Horo(g_iP_{(i)})$. Also we take the
vertical ray $\gamma_{+}\colon [0,\infty)\to \Xaug$ from
$q$ to parabolic point point $s_{j}$ of $\Horo(g_{j}P_{(j)})$. Then 
$\gamma_{-}([0,\infty))\cup l([0,a])\cup \gamma_{+}([0,\infty))$ 
is a bi-infinite geodesic from $s_i$ 
to $s_{j}$. 
There exists $h\in G$ such that $(i,g,n)=(i,hp,n)$. 
Then $(i,g,n)$ lies on the bi-infinite geodesic
$h(\gamma_{-}([0,\infty))\cup l([0,a])\cup \gamma_{+}([0,\infty)))$.
\end{proof}

Let $N_{\delta_0}$ be an integer greater than $\delta_0+1$. 
We fix $i\in \N$ and put 
$X^i:= \Gamma\cup \bigcup_{j\neq i}\Horo(g_jP_{(j)})$. Set
$e_i:=(i,g_i,N_{\delta_0})$ as in
remark~\ref{rem:vertices}.
There exists a metric $\rho_i$ on $\partial \Xaug$ 
which is compatible with the topology and satisfying that there
exists constants $A,C>0$ such that for any $u,v \in \partial \Xaug$, we
have $A^{-1}e^{-C(u|v)_{e_i}}\leq \rho_i(u,v) \leq Ae^{-C(u|v)_{e_i}}$. 


Let $s_i$ be the parabolic point of the combinatorial horoball 
$\Horo(g_iP_{(i)})$.
Set $\hat{P}_i := \partial \Xaug \setminus \{s_i\}$. We equip $\hat{P}_i$ with
the subspace topology, as a subspace of $\partial \Xaug$. 
Let $l\colon \R_{\geq 0} \to \Xaug$ be a geodesic ray such that
$l(0)=e_i$ and $l(\infty)\neq s_i$. We define
$n_i(l) := \max\{n:l(n)\in g_iP_{(i)}\}$.
By \cite[Lemma~3.10]{MR2448064}, we can assume that geodesic segments 
$l([0,\infty))\cap \Horo(g_iP_{(i)})$ consist of at most two vertical
segments and a single  horizontal segment of length at most 3. 




\begin{lemma}
\label{lem:find-geodesic}
 For any vertex $x\in X^i$, 
there exists a geodesic ray 
 $l_x\colon [0,\infty)\to \Xaug$ and $t_x\in [0,\infty)$
 such that $l_x(0) = e_i$, 
 $l_x(\infty) \neq s_i$, $l_x(t_x)\in X^i$
 and $d(x, l_x(t_x))\leq 2\delta_0$.
\end{lemma}

\begin{proof}
By Lemma~\ref{lem:taut}, there exists a 
 geodesic $l\colon (-\infty,\infty) \to \Xaug$ and 
$s_x\in (-\infty,\infty)$ such that $x=l(s_x)$.
Let $l_1,l_2\colon [0,\infty)\to \Xaug$ be geodesic rays such that
 $l_1(0)=l_2(0)=e_i$, $l_1(\infty)=l(-\infty)$, and 
$l_2(\infty)=l(\infty)$. We consider a geodesic triangle 
$l_1([0,\infty))\cup l((-\infty,\infty))\cup l_2([0,\infty))$. We can
 assume without loss of generality that $l(s_x)$ is contained in a 
$\delta_0$-neighborhood of $l_1([0,\infty))$. 
Therefore there exists 
$t_x'\in [0,\infty)$ such that $d(l(s_x),l_1(t_x'))\leq \delta_0$. 
Suppose that 
$l_1(\infty) = s_i$. Then 
$l_1([0,\infty))\subset \Horo(g_iP_{(i)};[N_{\delta_0},\infty))$, so $x$
 lies on the $\delta_0$-neighborhood of
 $\Horo(g_iP_{(i)};[N_{\delta_0},\infty))$. This contradicts 
 that $N_{\delta_0}>\delta_0$. Thus $l_1(\infty)\neq s_i$.
Set $l_x:= l_1$. Then we have $d(x,l_x(t_x'))\leq \delta_0$.
If $l_x(t_x')\in X^i$, then set $t_x := t_x'$, otherwise set
 $t_x:=n_i(l_x)$. 
Then $d(x,l_x(t_x))\leq 2\delta_0$. 
\end{proof}

In the rest of this section, we fix the following notations.
For any vertex $x\in X^i$, 
we choose a geodesic ray $l_x$ and $t_x\in [0,\infty)$ satisfying the
statement of Lemma~\ref{lem:find-geodesic}. 
For any point $u\in \hat{P}_i$, we
choose a geodesic ray $l^u$ such that $l^u(0)=e_i$ and $u=[l^u]$.  
%
\begin{lemma}
\label{lem:txsx}
 Let $x\in X^i$ be a vertex. Set $u=[l_x]$. There exists $s_x\in [0,\infty)$ 
 such that $l^u(s_x)\in X^i$ and $d(x,l^u(s_x))\leq 4\delta_0$.
\end{lemma}
\begin{proof}
 The Hausdorff distance of $l_x$ and $l^u$ is at most $\delta_0$. Thus
 there exists $s'_x\in [0,\infty)$ such that 
$d(l_x(t_x),l^u(s'_x))\leq \delta_0$. If $l^u(s'_x) \in X^i$, we put 
$s_x = s'_x$, otherwise we put $s_x=n_i(l_u)$. Then by
 Lemma~\ref{lem:find-geodesic}, $d(x,l^u(s_x)) \leq 4\delta_0$.
\end{proof}


\begin{lemma}
\label{lem:geodesic-contraction}
 Let $l_1\colon [0,a]\rightarrow \Xaug$ and $l_2\colon [0,b]\rightarrow
 \Xaug$ be geodesics such that $l_1(0) = l_2(0) = e_i$, and both of 
$l_1(a)$ and $l_1(b)$ lie on $X^i$. Then
\[
 d(l_1(n_i(l_1)),l_2(n_i(l_2))) \leq d(l_1(a),l_2(b)) + 2\delta_0.
\]
\end{lemma}

\begin{proof}
Set $x:=l_1(a)$, $y:=l_2(b)$, $x':=l_1(n_i(l_1))$ and $y':=l_2(n_i(l_2))$.
Here we remark that $x',y'\in g_iP_{(i)}$.
Let $r$ be an integer such that $d((i,x',r),(i,y',r))\leq 1$. We choose
 $g_{xy}$ such that 
$d((i,x',r),(i,g_{xy},r)) = d((i,y',r),(i,g_{xy},r)) = 1$. 
Set $p:= (i,g_{xy},r)$. 
We define $[p,x]$ as a geodesic consisting of a
 horizontal edge $\{(i,g_{xy},r),(i,x',r)\}$, a vertical geodesic 
$[(i,x',r),(i,x',0)]$ and $l_1([n_i(l_1),a])$. We also define a geodesic $[p,y]$
 similarly. 
We consider a geodesic triangle 
$[p,x] \cup [x,y] \cup [p,y]$, which is $\delta_0$-thin. Here we remark
 that $d(p,x') = d(p,y')=r+1$. 
If $r+1\leq (x|y)_p$, 
then 
\[
 d(x', y') \leq \delta_0.
\]
If $r+1> (x|y)_p$, then $d(x,x')\leq (y|p)_x$ since 
$(x|y)_p + (y|p)_x = d(p,x)$. Therefore, for a point $z\in [x,y]$ with
 $d(x,z) = d(x,x')$, we have $d(x',z) \leq \delta_0$.
By the same reason, for a point $w\in [x,y]$ with $d(y,w)=d(y,y')$, 
we have $d(y',w)\leq \delta_0$.
Since $d(z,w)\leq  d(x,y)$, we have
\[
  d(x',y') \leq d(x,y) + 2\delta_0.
\]
\end{proof}
We define a map $L_i\colon \hat{P}_i \to g_iP_{(i)}$ and 
$F_i\colon X^i\to \hat{P}_i$ as follows:
\begin{align*}
L_i(u) &:= l^u(n_i(l^u)) \textrm{ for } u\in \hat{P}_i;\\
F_i(x) &:= [l_x]\textrm{ for } x\in X^i.
\end{align*}
\begin{lemma}
 \label{lem:close}
 For any $x\in g_iP_{(i)}$, we have $d(x,L_i(F_i(x)))\leq 6\delta_0$.
\end{lemma}

\begin{proof}
 Let $x\in g_iP_{(i)}$. Set $u=[l_x]$. By Lemma~\ref{lem:txsx}, 
there exists $s_x\in [0,\infty)$ such that 
$d(x,l^u(s_x)) \leq 4\delta_0$ and $l^u(s_x)\in X^i$. Then by
 Lemma~\ref{lem:geodesic-contraction}, 
$d(x, L_i([l_x])) \leq d(x,l^u(s_x)) + 2\delta_0 \leq 6\delta_0$.
\end{proof}

\begin{lemma}
\label{lem:large-scale-Lipshitz}
 The composite $L_i\circ F_i$ is a large scale Lipschitz map, in fact, 
 for any $x,y \in X^i$, we have 
\[
 d(L_i\circ F_i(x),L_i\circ F_i(y))\leq d(x,y) + 10\delta_0.
\]
\end{lemma}

\begin{proof}
Let $x,y\in X^i$. 
Set $u = [l_x]$ and $v=[l_y]$. Then $L_i\circ F_i(x) = l^u(n_i(l^u))$
 and $L_i\circ F_i(y) = l^v(n_i(l^v))$. By Lemma~\ref{lem:txsx}, 
there exist $s_x,s_y>0$ such that $d(x,l^u(s_x))\leq 4\delta_0$ and 
 $d(y,l^v(s_y))\leq 4\delta_0$. Then by 
Lemma~\ref{lem:geodesic-contraction},
\[
d(L_i\circ F_i(x),L_i\circ F_i(y))\leq d(x,y) + 10\delta_0.
\]
\end{proof}

We equip $\hat{P}_i$ with the pullback coarse structure by $L_i$. We
remark that $\hat{E} \subset \hat{P}_i \times \hat{P}_i$
is controlled if and only if there exists $R>0$ such that for any 
$(u,v) \in \hat{E}$, we have $d(L_i(u),L_i(v)) < R$. 
\begin{lemma}
\label{lem:going-out}
Let $l\colon [0,\infty)\rightarrow \Xaug$ be a geodesic such that $l(0)
 =e_i$ and $l(\infty) \neq s_i$. Then for any $r>0$, there exists $t_r$
 such that for all $t\geq t_r$, we have $d(l(t),\Horo(g_iP_{(i)})) > r$.
\end{lemma}

\begin{proof}
Suppose that there exists $r>0$ such that $d(l(t),\Horo(g_iP_{(i)}))
 \leq r$ for all $t\geq 0$. Since the $r$-neighborhood of
 $\Horo(g_iP_{(i)})$ is coarsely equivalent to $\Horo(g_iP_{(i)})$,
 by 
 Proposition~\ref{lem:parabolic-pt-of-horoball}, $l(t)$ converges to a
 parabolic point $s_i$ as $t$ goes to infinity. This contradicts the
 assumption. 
\end{proof}

\begin{lemma}
\label{lem:P_i-is-proper-coarse-space}
 $\hat{P}_i$ is a proper coarse space.
\end{lemma}
\begin{proof}
We show that $L_i$ satisfies the conditions in
 Proposition~\ref{prop:pullback-proper}.
Let $K\subset g_iP_{(i)}$ be a compact set. Fix $R> 0$ 
such that $K\subset B(e_i;R)$. Here $B(e_i;R)$
 denotes a closed ball in $\Xaug$ of radius $R$ centered at $e_i$. 
Let $u\in \hat{P}_i$. 
If $L_i(u) \in B(e_i;R)$, 
then $(u|s_i)_{e_i}\leq R$. 
Therefore we have
\begin{align*}
 L_i^{-1}(K) &\subset \{u \in \hat{P}_i:d(e_i,L_i(u))\leq R\}\\
 &\subset \{ u \in \hat{P}_i: 
\rho_i(s_i,u)\geq A^{-1}e^{-CR}\}.
\end{align*}
Thus $L_i^{-1}(K)$ is relatively compact. 

We fix $u\in \hat{P}_i$. Since $u\neq s_i$, by
 Lemma~\ref{lem:going-out}, there exists $t_0>0$ such that for all $t>t_0$,
we have $d(l^u(t),\Horo(g_iP_{(i)})) > 2\delta_0$.
Let $v\in \hat{P}_i$ such that $(u|v)_{e_i}> t_0 + 3\delta_0$. 
By Lemma~\ref{lem:approx}, there exists $s>0$ such that 
$(l^u(s),l^v(s))_{e_i}\geq (u|v)_{e_i} -3\delta_0$. 
Set $\tau=(l^u(s)|l^v(s))_{e_i}$. Since $\tau > t_0$, 
we have $d(l^u(\tau),\Horo(g_iP_{(i)}))>2\delta_0$.
Since a geodesic triangle 
\[
 l^u([0,s])\cup  [l^u(s),l^v(s)]\cup l^v([0,s])
\]
is $\delta_0$-thin, we have
$d(l^u(\tau),l^v(\tau))\leq \delta_0$. Thus, 
$d(l^v(\tau),\Horo(g_iP_{(i)}))>\delta_0$. Then we can apply 
Lemma~\ref{lem:geodesic-contraction} to $l^u$ and $l^v$, so we have
$d(L_i(u),L_i(v)) < 3\delta_0$. 
Thus, the inverse image $L_i^{-1}(B(L_i(u),3\delta_0))$ contains a
 neighborhood 
$\{v\in \hat{P}_i: (u|v)_{e_i}> t_0+3\delta_0\}$
of $u$.
Therefore $L_i$ is pseudocontinuous. 
\end{proof}

\begin{proposition}
\label{prop:coarse_equivalence}
 $\hat{P}_i$ and $g_iP_{(i)}$ are coarsely equivalent.
\end{proposition}

\begin{proof}
 We define a map $H_i\colon g_iP_{(i)} \to \hat{P}_i$ 
 as the restriction of $F_i$, that is,  $H_i(x) := [l_x]$ for 
 $x\in g_iP_{(i)}$. Then by Lemma~\ref{lem:close}, 
 the composite $L_i\circ H_i$ is close to the
 identity. So by Proposition \ref{prop:pull_back_coarse_str}, 
 $\hat{P}_i$ and $g_iP_{(i)}$ are coarsely equivalent.
\end{proof}

\begin{proposition}
\label{prop:Pullback}
 For any Higson function $f\in C_h(\hat{P}_i)$, the pullback 
$F_i^*f := f\circ F_i$ is a Higson function on $X^i$.
\end{proposition}
\begin{proof}
Let $f\in C_h(\hat{P}_i)$ be a Higson function. We fix $\epsilon>0$ 
and $R>0$. 
Let $\hat{E}:=\{(u,v):d(L_i(u),L_i(v))< R+10\delta_0\}$ 
be a controlled set of $\hat{P}_i$. There exists $S>0$ such that for a
bounded set $\hat{K} := \{u\in \hat{P}_i:d(e_i, L_i(u))<S\}$ and for any
$(u,v)\in \hat{E}$, 
\begin{align}
\label{eq:Higson_condition}
 \text{if } (u,v) \notin \hat{K}\times \hat{K} 
 \text{ then } \abs{\grad f(u,v)} < \epsilon.
\end{align}
On the other hand,
since $\hat{P}_i$ is a proper coarse space, $\hat{K}$ is relatively compact.
Thus the restriction $f|\hat{K}$ is uniformly continuous on $\hat{K}$, so
there exists $\theta>0$ such that
\begin{align}
\label{eq:uniformly_continuous}
 \text{for any } u,v \in \hat{K}, \text{ if } \rho_i(u,v)< \theta 
 \text{ then } \abs{\grad f(u,v)}<\epsilon.
\end{align}

Let $E_R := \{(x,y):d(x,y)<R\}$ be a controlled set of 
$X^i$. 
By Lemma~\ref{lem:large-scale-Lipshitz}, we have
 $F_i(E_R) \subset \hat{E}$. 
Set 
\begin{align*}
K'&:=\{x\in X^i: d(e_i,L_i\circ F_i(x))<S\};\\
T & := -\frac{1}{C}\log(\frac{\theta}{A}) + R + 4\delta_0;\\
K & := B(e_i,T).
\end{align*}
We remark that $K'$ is unbounded.
Let $(x,y) \in E_R$ such that $(x,y)\notin K\times K$.
We first assume $(x,y)\notin K'\times K'$, then 
$(F_i(x),F_i(y)) \notin \hat{K}\times\hat{K}$.
Thus by (\ref{eq:Higson_condition}) we have
$\abs{\grad F_i^*f(x,y)} = \abs{\grad f(F_i(x),F_i(y))}< \epsilon$.
Next, we assume $(x,y) \in K'\times K'$.
Since Lemma~\ref{lem:find-geodesic} implies
\[
 ([l_x]|[l_y])_{e_i}\geq (l_x(t_x)|l_y(t_y))_{e_i}\geq 
 T-R-4\delta_0,
\] we have
$\rho_i([l_x],[l_y]) < Ae^{-C(T-R-4\delta_0)}= \theta$. 
 Then by (\ref{eq:uniformly_continuous}) 
we have $\abs{\grad F_i^*f(x,y)} < \epsilon$.
\end{proof}

By Proposition~\ref{prop:Pullback}, $F_i$ extends to a continuous
map 
\[
 hF_i\colon hX^i\to h\hat{P}_i. 
\]

Since the Gromov boundary is a corona, there exists a
continuous map 
\[
 \alpha \colon h\Xaug \to \barXaug
\]
which is the identity on $\Xaug$. Since
coarse embedding $X^i\hookrightarrow \Xaug$ induces an embedding 
$\nu X^i\hookrightarrow \nu \Xaug$, we regard $\nu X^i$ as a subspace of
$\nu \Xaug$. (See Proposition~\ref{prop:Dra-embedding}.)
\begin{lemma}
\label{lem:F_and_alpha}
 For any $y \in \nu X^i$, if $y \notin \alpha^{-1}(s_i)$
 then we have $\alpha(y) = hF_i(y)\in \hat{P}_i$.
\end{lemma}
\begin{proof} Let $y \in \nu X^i\setminus \alpha^{-1}(s_i)$.
 We choose a net $\{y_\lambda\}_{\lambda \in \Lambda}$ in 
$X^i$ such that 
$y_\lambda \to y$. Then $\alpha(y_\lambda) \to \alpha(y)$.
The restriction of $\alpha$ to $\Xaug$ is the identity, so
\[
 (F_i(y_\lambda)|\alpha(y))_{e_i} \geq (y_\lambda|\alpha(y))_{e_i}
= (\alpha(y_\lambda)|\alpha(y))_{e_i} \to \infty.
\]
Thus $F_i(y_\lambda) \to \alpha(y)$ in $\hat{P}_i$, so
we have $hF_i(y)=\alpha(y)$.
\end{proof}

\subsection{Blow-up of parabolic points}
\label{sec:blow-up-parabolic}


In this section, we construct a corona of
\begin{align*}
 X_n &= \Gamma\cup \bigcup_{i>n} \Horo(g_iP_{(i)}).
\end{align*}
For $r= 1,\dots,k$, let $(W_r,\zeta_r)$ be a corona of $P_{r}$. 
For $i\in \N$, set $W_i := W_{(i)}$ and 
$\zeta_i:= \zeta_{(i)}\circ \nu g_i^{-1}$, where 
$\nu g_i^{-1}\colon \nu (g_iP_{(i)})\to \nu P_{(i)}$ is an homeomorphism
induced by an isometry $g_iP_{(i)}\ni x\mapsto g_i^{-1}x\in P_{(i)}$.
Then $(W_i,\zeta_i)$ is a corona of $g_iP_{(i)}$.
By Proposition~\ref{prop:coarse_equivalence}, 
$\nu \hat{P}_i$ is homeomorphic to $\nu g_iP_{(i)}$, so we identify these two
spaces. Thus we have a corona 
$(W_i,\zeta_i)$ of $\hat{P_i}$ and a compact metrizable space
$\hat{P_i} \cup W_i$. We recall that
$\bar{\zeta_i}\colon h\hat{P_i}\rightarrow \hat{P_i}\cup W_i$ denotes an
extension of $\zeta_i$ by the identity on $\hat{P_i}$.
(See Section~\ref{sec:higs-comp}.)

We construct a corona of $X_n$ by replacing $s_i$ by $W_i$ as follows.
Set 
\begin{align}
 \partial  X_n(W_i;i=1,\dots,n) &:= \partial \Xaug \setminus \{s_1,\dots,s_n\} 
 \sqcup \bigsqcup_{i=1}^n W_i.
\end{align} 
We abbreviate $\partial X_n(W_i;i=1,\dots,n)$ to $\partial X_n$.
We equip $\partial X_n$ with the weakest topology such that 
the maps $\sigma_i\colon \partial X_n\to \hat{P_i}\cup W_i$ are
continuous for all $i\in \{1,\dots,n\}$. Here $\sigma_i(x)=s_j$ if $x\in
W_j$ with $j\neq i$ and $\sigma_i(x) = x$ otherwise.




\begin{definition}
 The $n$-th blown-up of $\partial \Xaug$ with respect to 
$W_i, i=1,\dots,n$ is a compact space $\partial X_n(W_i;i=1\dots,n)$
 equipped with the above topology. 
 The blown-up corona of
 $(G,\famP,\{W_1,\dots,W_k\})$ is the
 projective limit $\dX_\infty = \varprojlim \dX_n$.
\end{definition}

We also regard $\nu X_n$ and $\nu G$ as subspaces of $\nu \Xaug$.
We define a map $\xi_n \colon \nu X_n \to \partial X_n$ as
\begin{align*}
 \xi_n(x) := \begin{cases}
	   \alpha(x)& \text{ if } x\notin
	   \bigcup_{i=1}^n\alpha^{-1}(s_i),\\
	   \bar{\zeta}_i \circ hF_i(x)& \text{ if } x \in
	   \alpha^{-1} (s_i) \text{ for } i=1,\dots,n.
	  \end{cases}&
\end{align*}
\begin{proposition}
 The map $\xi_n \colon \nu X_n \to \partial X_n$ is
 continuous for all $n\in \N\cup \{\infty\}$. Thus $\partial X_n$ and
 $\partial X_\infty$ are
 respectively coronae of $X_n$ and $G$. If $\zeta_i$ is surjective for
 $i=1,\dots,k$, then so is $\xi_n$ for all $n\in \N\cup \{\infty\}$.
\end{proposition}

\begin{proof}
It is enough to show that $\xi_{n}$ is continuous on 
$\nu X_n \cap \alpha^{-1}(s_i)$. 
We fix $x\in \nu X_n \cap \alpha^{-1}(s_i)$. Let $\{x_\lambda\}_{\lambda \in
 \Lambda}$ 
be a net in $\nu X_n$ such that 
$x_\lambda \to x$. If $x_\lambda \in \alpha^{-1}(s_i)$ then
 $\xi_{n}(x_\lambda) = \bar{\zeta}_i \circ hF_i(x_\lambda)$.
If $x_\lambda \notin \alpha^{-1}(s_i)$ then by
 Lemma~\ref{lem:F_and_alpha}, $\xi_{n}(x_\lambda) = \alpha(x_\lambda) =
 \bar{\zeta}_i \circ hF_i(x_\lambda)$. Here we remark that $\bar{\zeta}_i$ is the
 identity on $\hat{P_i}$. Since $\bar{\zeta_i}\circ hF_i$ is continuous,
we have $\xi(x_\lambda) \to \xi(x)$.

We suppose $\zeta_r$ is surjective for $r=1,\dots,k$. We show that
 $\xi_n$ is surjective for all $n\in \N$. In fact, we prove that the
 restriction $\xi_n\colon \nu G\to \partial X_n$ is  surjective. 
Since the action of $G$ on $\partial X_0 = \partial \Xaug$
 is minimal (\cite[Section 6]{MR2922380}), 
$\xi_0\colon \nu G \to \partial X_0$ is surjective. 
We assume that $\xi_n\colon \nu G \to \partial X_n$ is surjective. 
Let $\pi_n\colon \partial X_{n+1}\to \partial X_n$ be a natural
 projection. Then we have $\xi_{n}=\pi_n\circ \xi_{n+1}$. 
Let $x\in \partial X_{n+1}$. If $x\in \pi_n^{-1}(s_{n+1})=W_{n+1}$, 
then there exists $y$ in $\nu (g_{n+1}P_{(n+1)})$ such that 
$\xi_{n+1}(y)=x$, where we regard $\nu (g_{n+1}P_{(n+1)})$ as a subspace
 of $\nu G$. Otherwise, there exists $y'\in \nu G$ such that 
$\pi_n(x)=\xi_n(y')=\pi_n(\xi_{n+1}(y'))$. Then we have 
$\xi_{n+1}(y') = x$ since the restriction of $\pi_{n}$ to the complement of
 $\pi_n^{-1}(s_{n+1})$ is injective.

%
\end{proof}

\section{The transgression maps}
\label{sec:transgression-maps}
Let $M^*$ be the
$K$-theory $K^*$ or the Alexander-Spanier cohomology with compact
support $H_c^*$.
Let $G$ be a group which is hyperbolic relative to $\famP$ satisfying the
condition of Theorem~\ref{main_theorem}. Let $\{g_n\}_{n\in\N}$ be an order
of the cosets of $(G,\famP)$. Let $s_i$ be the parabolic point of $g_iP_{(i)}$.
Let $X_n$ and $EX_n$ be as defined in
Section~\ref{subsec:weak-coars-relat}. We can choose a map 
$\varphi_n \colon EX_n \to X_n$ such that the pullback coarse
structure is proper and the $\varphi_n$ is a coarse equivalence. (See
loc. cit.) 
Therefore we can regard a corona of $X_n$ as that  
of $EX_n$.

For a compact space $Z$, we denote by $\mathcal{C}Z$ 
a closed cone of $Z$,  that is, 
$\mathcal{C}Z = Z \times [0,1]/\sim$ where 
$(z,1)\sim (z',1)$ for all $z,z'$ in $Z$. 
Let $(W_i,\zeta_i)$ be a corona of $g_iP_{(i)}$ as in
Section~\ref{sec:blow-up-parabolic}.
Let $\dX_n = \partial  X_n(W_i;i=1,\dots,n)$ be the $n$-th blown-up 
of $\partial \Xaug$.
Let $S_n$ be a space obtained by pasting $\mathcal{C}W_{n+1}$ on $\dX_{n+1}$ 
along $W_{n+1}$.
\[
 S_n := \dX_{n+1} \cup \mathcal{C}W_{n+1}.
\]

\begin{lemma}
\label{lem:homotopy_cone}
The natural quotient map $S_n \to \dX_{n}$ which sends
 $\mathcal{C}W_{n+1}$ to the parabolic point $s_{n+1}$ induces an isomorphism 
 $M^*(\dX_{n}) \cong M^*(S_n)$.
\end{lemma}
\begin{proof}
 The lemma follows from the strong excision property. 
(See \cite[Chapter 6, Section 6]{MR666554} 
for the case of Alexander-Spanier cohomology.)
\end{proof}
We use the following notations.
\begin{align*}
E\Horo(g_iP_{(i)}) &:= g_iEP_{(i)} \times [0,\infty);\\
\overline{E\Horo(g_iP_{(i)})} &:= \mathcal{C}(g_i\EP_{(i)}\cup_{\zeta_i} W_i).
\end{align*}
Then $\overline{E\Horo(g_iP_{(i)})}$ is a compactification of 
$E\Horo(g_iP_{(i)})$ and 
$\overline{E\Horo(g_iP_{(i)})} \setminus E\Horo(g_iP_{(i)}) =
\mathcal{C}W_i$.
We remark that $\overline{E\Horo(g_iP_{(i)})}$ is not any coarse
compactification of $E\Horo(g_iP_{(i)})$. 


\begin{proposition}
\label{prop:boundary-map-for-Xn}
We suppose that the boundary map 
$\partial \colon \tilde{M}^{*-1}(W_i)\to M^*(\EP_i)$ is an isomorphism
 for $i=1,\dots,k$. Then 
 $\partial \colon \tilde{M}^{*-1}(\dX_n) \to M^*(EX_n)$ is an isomorphism
 for all $n\geq 1$.
\end{proposition}

\begin{proof}
  Since $\Xaug$ is hyperbolic and $\partial \Xaug$ is its Gromov
 boundary, 
by Corollary~\ref{cor:transgression} and
 Lemma~\ref{lem:homotopy_cone}, the
 boundary map induces an isomorphism
 \[
  \tilde{M}^{*-1}(S_0) \cong M^*(\EX).
 \]
The proposition inductively follows 
from Lemma~\ref{lem:homotopy_cone} and Mayer-Vietoris sequences
 for $S_n = \dX_{n+1} \cup \mathcal{C}W_{n+1}$ and for 
 $EX_n = EX_{n+1}\cup E\Horo(g_{n+1}P_{(n+1)})$:
\begin{align*}
 \xymatrix{
\ar[r] & \tilde{M}^{q-1}(S_n) \ar[d] \ar[r] 
& \tilde{M}^{q-1}(\dX_{n+1}) \oplus \tilde{M}^{q-1}(\mathcal{C}W_{n+1}) \ar[d] \ar[r] 
& \tilde{M}^{q-1}(W_{n+1})  \ar[r] \ar[d]&\\
\ar[r] & M^q(EX_n) \ar[r] 
& M^q(EX_{n+1}) \oplus M^q(E\Horo(g_{n+1}P_{(n+1)})) \ar[r]
& M^q(g_{n+1}\EP_{(n+1)}) \ar[r] &.
}
\end{align*}
Here we remark that 
$\tilde{M}^{q-1}(\mathcal{C}W_{n+1}) = M^q(E\Horo(g_{n+1}P_{(n+1)})) = 0$.
\end{proof}

\subsection{Proof of Theorem~\ref{main_theorem}}
Let $M^*$ be the Alexander-Spanier cohomology with compact supports or
the $K$-theory. Let $(W_r,\zeta_r)$ be a corona of $P_r$ for
$r=1,\dots,k$. We remark that the boundary map 
$\partial \colon \tilde{M}^{*-1}(W_i) \to M^*(g_i\EP_{(i)})$ 
is an isomorphism if and only if so is the transgression map 
$T_{W_i}\colon \tilde{M}^{*-1}(W_i) \to MX^*(g_iP_{(i)})$. A similar
statement for $K$-homology holds. By the continuity of $M^*$, we have
$\tilde{M}^{*-1}(\dX_\infty)\cong \varinjlim \tilde{M}^{*-1}(\dX_n)$.
Therefore, if $T_{W_r} \colon \tilde{M}^{*-1}(W_r) \to MX^*(P_r)$ is
an isomorphism for all $r=1,\dots,k$, then by
Proposition~\ref{prop:boundary-map-for-Xn} and
Corollary~\ref{cor:limX_n=G}, 
we have $\tilde{M}^{*-1}(\dX_\infty)\cong MX^*(G)$. 

If $T_{W_r} \colon KX_*(P_r) \to \tilde{K}_{*-1}(W_r)$ is an isomorphism for all
$r=1,\dots,k$, then, by the same way as in the proof of
Proposition~\ref{prop:boundary-map-for-Xn}, we can show that 
$K_*(EX_n)\cong \tilde{K}_{*-1}(\dX_n)$ for all $n\in \N$.
By the Milnor exact sequence for 
$K_*(EX_n)$ and $K_{*-1}(\dX_n)$, we have 
$KX_*(G) \cong \tilde{K}_{*-1}(\dX_\infty).$ 

\section{Application}\label{sec:application}
We give two applications of Theorems~\ref{th:coarse_co-assembly-map} and \ref{main_theorem}. 
First, we consider virtually polycyclic groups. We recall the following fact
\cite[Proposition 4.4]{MR2379803}. 
\begin{theorem}[Ji]\label{nilnil}
Any virtually polycyclic group $P$ has a finite $P$-simplicial complex $\underline{E}P$
which is a universal space for proper $P$-actions. 
\end{theorem}

It follows from \cite[Theorem 1.1]{MR1728880}, \cite[Theorem 9.2]{MR2225040}, 
Theorem~\ref{nilnil}
and the fact that any virtually polycyclic group has Yu's Property A that the coarse
assembly map and the coarse co-assembly map for the group are isomorphisms.
\begin{proposition}\label{solv}
Let $P$ be a virtually polycyclic group.
Then there exists a corona $W$ of $P$ such that 
$W$ is homeomorphic to a sphere $S^{n-1}$ 
and satisfies the following: 
\begin{align*}
K_*(C^*(P))\cong KX_*(P)\cong \tilde{K}_{*-1}(W)\cong 
\left\{
\begin{array}{cc}
\mathbb Z   & (*=n)   \\
0  & (*= n+1)   \\
\end{array}
\right.,\\
K_{*-1}(\mathfrak c^r(P))\cong KX^*(P)\cong \tilde{K}^{*-1}(W)\cong \left\{
\begin{array}{cc}
\mathbb Z   & (*=n)   \\
0  & (*=n+1  )   \\
\end{array}
\right.,\\
HX^*(P)\cong \tilde{H}^{*-1}(W)\cong \left\{
\begin{array}{cc}
\mathbb Z   & (*=n)   \\
0  & (*\neq n)   \\
\end{array}
\right. .
\end{align*}
\end{proposition}
\begin{proof}
Any virtually polycyclic group has a finite index subgroup 
which is isomorphic to a polycyclic group by definition. 
It follows from \cite[Theorem 4.28]{MR0507234} that 
any polycyclic group has a finite index normal subgroup which is isomorphic to a lattice of 
some $n$-dimensional simply connected solvable Lie group. 
Hence the given virtually polycyclic group is naturally coarsely 
equivalent to a lattice of some $n$-dimensional simply connected solvable Lie group. 

Now we can assume that a given group $P$ is 
a lattice of some $n$-dimensional simply connected solvable Lie group $H$ 
without loss of generality.
Then it follows from the Mayer-Vietoris argument in \cite[Section 7]{product-cBC} that 
the coarse assembly map and the coarse co-assembly map for the group are isomorphisms.


By \cite[Section 7]{product-cBC}, $H$ has a coarse compactification 
$H\cup W$ which is homeomorphic to the closed ball in $n$-dimensional euclidean space. 
Moreover $W$ is homeomorphic to $S^{n-1}$. 

Since $H$ is uniformly contractible and has bounded geometry, the
coarsening map and the character maps 
\[
K_*(H)\to KX_*(H),\, KX^*(H)\to K^*(H),\, HX^*(H)\to H^*_c(H)
\]
are isomorphisms. (See \cite[Section 3]{MR1388312},
\cite[Theorem 4.8]{MR2225040}, \cite[(3.33) Proposition]{MR1147350}). 
Also since $\overline{H}$ is contractible, we have 
\[
K_*(H)\cong \tilde{K}_{*-1}(W), \;
\tilde{K}^{*-1}(W)\cong K^*(H), \;
\tilde{H}^{*-1}(W)\cong H^*_c(H).
\]
Hence we have 
\[
 KX_*(H)\cong \tilde{K}_{*-1}(W),\;
 \tilde{K}^{*-1}(W)\cong KX^*(H),\;
 \tilde{H}^{*-1}(W)\cong HX^*(H).
\]

Since the inclusion from $P$ to $H$ is a coarse equivalence map, 
$W$ is regarded as a corona of $P$ and thus the map 
covers the identity on $W$, we have the assertion. 
\end{proof}

\subsection{Coronae of the fundamental groups of  pinched negatively curved complete Riemannian manifolds with finite volume}

\begin{corollary}\label{pinched}
Let $G$ be a group which properly isometrically acts on 
an $m$-dimensional simply-connected pinched negatively curved complete 
Riemannian manifold $X$.
Suppose that the quotient is with finite volume, but not compact.  
Then we have a corona $\partial G$ of $G$ and the following: 
\begin{align*}
K_*(C^*(G))\cong KX_*(G)\cong \tilde{K}_{*-1}(\partial G)\cong 
\left\{
\begin{array}{cc}
\prod_{i\in\mathbb N} \mathbb Z   & (*=m-1)   \\
0  & (*= m)   \\
\end{array}
\right.,\\
K_{*-1}(\mathfrak c^r(G))\cong KX^*(G)\cong \tilde{K}^{*-1}(\partial G)\cong \left\{
\begin{array}{cc}
\bigoplus_{i\in\mathbb N} \mathbb Z   & (*=m-1)   \\
0  & (*= m )   \\
\end{array}
\right.,\\
HX^*(G)\cong \tilde{H}^{*-1}(\partial G)\cong \left\{
\begin{array}{cc}
\bigoplus_{i\in\mathbb N} \mathbb Z   & (*=m-1)   \\
0  & (*\neq m-1)   \\
\end{array}
\right. .
\end{align*}
\end{corollary}

\begin{proof}
It is already known that 
the coarse assembly map and the coarse co-assembly map for
$G$ in the above are isomorphisms. Indeed
$G$ is known to be hyperbolic relative to a family of virtually
nilpotent subgroups (\cite[8.6]{asym_invs} and also \cite[Theorem 5.1]{Far98}) 
and thus we can use
\cite[Theorem 1.1]{MR1728880}, \cite[Theorem 9.2]{MR2225040} 
and \cite[Section 1]{MR2364071}.
This fact also follows from Proposition~\ref{solv}, Theorem \ref{nilnil} 
and Theorem~\ref{th:coarse_co-assembly-map}.
Note that nilpotent groups are polycyclic groups.

We take a set $\mathbb P$ of representatives of conjugacy invariant
classes of maximal parabolic subgroups of $G$ with respect to the action on $X$.  
Then $\mathbb P$ is a finite family of virtually nilpotent groups, 
and $G$ is hyperbolic relative to $\mathbb P$ (\cite[8.6]{asym_invs} and also \cite[Theorem 5.1]{Far98}).  
Then we have that 
$\mathbb P$ satisfies the assumptions in
Theorem~\ref{main_theorem} by
Proposition \ref{solv} and Theorem \ref{nilnil}.
Indeed we take a corona $W_r$ of $P_r$ in Proposition \ref{solv}, 
which is homeomorphic to $S^{m-2}$ and satisfies 
\begin{align*}
KX_*(P_r)\cong \tilde{K}_{*-1}(W_r), 
KX^*(P_r)\cong \tilde{K}^{*-1}(W_r),
HX^*(P_r)\cong \tilde{H}^{*-1}(W_r).
\end{align*}
We define $\partial G$ as the
blown-up boundary of $(G, \mathbb P,\{W_r\})$.
Then Theorem~\ref{main_theorem}
implies the assertion except for concrete computations.

Now we compute $\widetilde{K}_{*}(\partial G)$. From now on, we refer to 
Section~\ref{sec:blow-up-parabolic} for symbols as $\partial X_n$ and so on.
Note that $\partial G$ is $\partial X_\infty=\underleftarrow{\lim}\partial X_n$
as in Proof of Theorem~\ref{main_theorem}. 
In order to use the Milnor exact sequence, 
we compute the map $\widetilde{K}_m(\partial X_{n+1}) \to \widetilde{K}_m(\partial X_{n})$
for any $n\in \mathbb N$. 

Note that the Gromov boundary $\partial_G X$ of $X$ is the Bowditch 
boundary of $(G,\mathbb P)$ and homeomorphic to a sphere $S^{m-1}$.
Take a finite generating set $\mathcal{S}$ of $G$.
Then we have a $G$-equivariant homeomorphism 
$\partial X(G,\mathbb P,\mathcal{S})\cong \partial_G X$
by uniqueness of the Bowditch boundary of a relatively hyperbolic group
(see \cite[Section 9]{MR2922380}). 
We note that $\mathbb P$ is not empty
because the action of $G$ on $X$ is not cocompact. 

We consider the following long exact sequence 
for the excision pair $(\partial X_n, W_n)$ for any $n\in \mathbb N$:
\begin{align}\label{ddd}
\to \widetilde{K}_*(W_n) \to \widetilde{K}_*(\partial X_n) 
\to K_*(\partial X_n\setminus W_n) \to \widetilde{K}_{*-1}(W_n) \to ,
\end{align}
where we put $\partial X_0:=\partial X(G,\mathbb P,\mathcal{S})\cong \partial_G X\cong S^{m-1}$. 
Note that $\partial X_n\setminus W_n$ is naturally homeomorphic to 
$\partial X_{n-1}\setminus \{s_n\}$ and also that 
$K_*(\partial X_n\setminus W_n)$ is naturally isomorphic to $\widetilde{K}_{*}(\partial X_{n-1})$.
For $n=1$, the boundary map of the long exact sequence (\ref{ddd}) is the composite of 
the coarsening map $K_*(\partial X_0\setminus \{s_1\})\to KX_*(\partial X_0\setminus \{s_1\})$ and 
the transgression map $KX_*(\partial X_0\setminus \{s_1\})\to \widetilde{K}_{*-1}(W_1)$, 
where $\partial X_0\setminus \{s_1\}$ is coarsely equivalent to $g_1P_{(1)}$. 
The latter map is an isomorphism because the transgression map 
$KX_*(g_1P_{(1)})\to \widetilde{K}_{*-1}(W_1)$ is an isomorphism. 
Also the former map is an isomorphism because $\partial X_0\setminus \{s_1\}$
is uniformly contractible and with bounded geometry.  
Hence the boundary map is an isomorphism for $n=1$ 
and thus we have $\widetilde{K}_{*}(\partial X_{1})=0$. 
Then by using the long exact sequence (\ref{ddd}) inductively, for any $n\ge 2$, 
we have a split exact sequence: 
\begin{align*}
0 \to \widetilde{K}_m(W_{n+1})\cong \mathbb Z \to 
\widetilde{K}_m(\partial X_{n+1}) 
\to \widetilde{K}_m(\partial X_n) \cong \prod_1^{n-1}\mathbb Z \to 0 
\end{align*}
and $\widetilde{K}_{m-1}(\partial X_n)=0$. 
Now we can compute the reduced $K$-homology of $\partial G$ by the Milnor exact sequence. 
By a similar way, we can compute the reduced $K$-theory and reduced
cohomology of $\partial G$.
\end{proof}

\subsection{Coronae of the fundamental groups of $3$-dimensional closed manifolds}
We give coronae of the fundamental groups of $3$-dimensional closed manifolds.

\begin{corollary}
Let $G$ be the fundamental group of a $3$-dimensional closed manifold $M$.
Suppose that $G$ is infinite. 
Then we have a corona $\partial G$ of $G$ and the following: 
\begin{align*}
K_*(C^*(G))\cong KX_*(G)\cong \tilde{K}_{*-1}(\partial G),\\
K_{*-1}(\mathfrak c^r(G))\cong KX^*(G)\cong \tilde{K}^{*-1}(\partial G),\\
HX^*(G)\cong \tilde{H}^{*-1}(\partial G).
\end{align*}
\end{corollary}
\begin{proof}
The coarse Baum-Connes conjecture for $G$ is well-known.
For example, each group can be coarsely embeddable to a Hilbert space
and thus satisfies the conjecture by Yu's result. 
The below contains another proof.

If $M$ is not orientable, then the fundamental group of the double
covering of $M$ is contained in that of 
$M$ with index $2$ and thus those 
two groups are coarsely equivalent. 
We can assume that $M$ is orientable without loss of generality.

We take a prime decomposition $M=N_1\#N_2\#\cdots\#N_n$ 
and put $P_j:=\pi_1(N_j)$ for each $j\in \{1,2,\ldots,n\}$. 
Then $G$ is regarded as a free product $P_1\ast P_2\ast \cdots\ast P_n$. 
We remark that $N_j$ is orientable for each $j\in \{1,2,\ldots,n\}$
and that $N_j$ is not irreducible only if $N_j$ is diffeomorphic to 
$S^1\times S^2$ and 
thus $P_j$ is isomorphic to $\Z$. 
Without loss of generality, we can assume that 
there exists $0\le m\le n$ such that 
$P_j$ is infinite and not cyclic for each $j\le m$
and otherwise $P_j$ is finite or cyclic. 
If $m=0$, then $G$ is hyperbolic and the assertion follows from Higson-Roe's result \cite{MR1388312}. 
We can assume that $m\ge 1$. 
Then $G$ is hyperbolic relative to $\{P_1,\ldots, P_m\}$.

Now we take $j\in \{1,2,\ldots,m\}$. 
If $N_j$ is geometric, then the universal cover 
$\widetilde{N_j}$ is a universal space for free proper 
$P_j$-actions and coarsely equivalent to $P_j$. 
Then $\widetilde{N_j}$ is isometric to either of model spaces 
of $6$-geometry except for $S^3$ and $S^1\times S^2$ by choice of $j$.
Each of them has a coarse compactification $\widetilde{N_j}\cup W_j$
which is homeomorphic to a closed ball in $3$-dimensional euclidean space and then 
the corona $W_j$ is homeomorphic to a $2$-dimensional sphere. 
Indeed $Nil$ and $Sol$ are simply connected solvable Lie groups 
with a lattice and thus have such coarse compactifications (\cite[Section 7]{product-cBC}). 
When we consider $\R^3$, $\mathbb H^2\times \R$ and $\mathbb H^3$, 
they are Hadamard manifolds and thus the visual boundaries give such coarse compactifications 
(\cite{MR1388312}, \cite{WillettThesis}, \cite{Busemann_cBC}). 
Also since we have a homeomorphic coarse equivalence from 
$\widetilde{PSL}(2,\R)$ to $\mathbb H^3$
(see for example \cite[Section 2]{Kap-Leeb-3mdf-gr}), 
the visual boundary of $\mathbb H^3$ induces a desired coarse 
compactification of $\widetilde{PSL}(2,\R)$.
Then $P_j$ and $W_j$ satisfy assumptions in Theorems~\ref{th:coarse_co-assembly-map} and 
\ref{main_theorem}. 

If $N_j$ is not geometric, 
it follows from Thusrton's geometrization conjecture which was 
solved by Perelman that $N_j$ is a Haken manifold. 
Suppose that $N_j$ is Haken and not geometric.
Fix a metric on $N_j$.  
By Kapovich-Leeb's result \cite{Kap-Leeb-3mdf-gr}, even if $N_j$ 
is not non-positively curved and moreover has no metric with 
non-positive curvature, 
there exists a closed $3$-dimensional non-positively curved manifold $L_j$ 
and a bilipschits homeomorphism between the universal covers 
$\widetilde{N_j}$ and $\widetilde{L_j}$. 
In particular $P_j$ and $\widetilde{L_j}$ are coarsely equivalent. 
Since $\widetilde{L_j}$ is an Hadamard manifold, $\widetilde{L_j}$ and 
(thus $P_j$) has a coarse compactification which is homeomorphic to a 
closed ball in $3$-dimensional euclidean space.
The corona $W_j$ is homeomorphic to a $2$-dimensional sphere. 
Then $P_j$ and $W_j$ have assumptions in Theorems~\ref{th:coarse_co-assembly-map} and 
\ref{main_theorem}.

By Theorem~\ref{th:coarse_co-assembly-map} 
the coarse assembly map and the coarse co-assembly map for $G$ are isomorphisms.
Moreover by Theorem~\ref{main_theorem}, 
we have a desired corona $W$ of $G$.  
\end{proof}

\begin{corollary}\label{3-mfd}
Let $G$ be the fundamental group of a $3$-dimensional orientable closed manifold $M$.
Take a prime decomposition $M_1=N_1\#N_2\#\cdots\#N_m$. 
Suppose that $m$ is at least $2$, all fundamental groups of $N_j$ are infinite
and all $N_j$ are irreducible. 
Then we have a corona $\partial G$ of $G$ and the following: 
\begin{align*}
K_*(C^*(G))\cong KX_*(G)\cong \tilde{K}_{*-1}(\partial G)\cong 
\left\{
\begin{array}{cc}
\prod_{i\in\mathbb N} \mathbb Z   & (*=1)   \\
0  & (*= 0)   \\
\end{array}
\right.,\\
K_{*-1}(\mathfrak c^r(G))\cong KX^*(G)\cong \tilde{K}^{*-1}(\partial G)\cong \left\{
\begin{array}{cc}
\bigoplus_{i\in\mathbb N} \mathbb Z   & (*=1)   \\
0  & (*= 0 )   \\
\end{array}
\right.,\\
HX^*(G)\cong \tilde{H}^{*-1}(\partial G)\cong \left\{
\begin{array}{cc}
\bigoplus_{i\in\mathbb N} \mathbb Z   & (*=1)   \\
0  & (*\neq 1)   \\
\end{array}
\right. .
\end{align*}
\end{corollary}
\begin{proof}
Since $G$ is isomorphic to $P_1*\cdots *P_m$, 
the group $G$ is hyperbolic relative to $P_1,\ldots,P_m$ and 
the Bowditch boundary is homeomorphic to the Cantor set. 
Also we take a corona of $P_j$ in the above proof, 
which is homeomorphic to $S^2$.
Then we can compute the reduced $K$-homology, the reduced $K$-theory 
and reduced cohomology of a blown-up corona $\partial G$ of $G$ 
by a similar way as Proof of Corollary \ref{pinched}.
\end{proof}

\appendix
\section{Milnor exact sequences by Phillips}
\label{appendix:Milnor}
The $K$-theory for $C^\ast$-algebras can be extended
for countable projective limits of $C^\ast$-algebras
that are called $\sigma$-$C^\ast$-algebras.
Phillips \cite{MR1050490} studied such an extended theory
that he called the representable $K$-theory.
The theory possesses basic properties of the ordinary $K$-theory.
Indeed the theory consists of functors $RK_i$
from the category of $\sigma$-$C^\ast$-algebras to the category of
abelian groups,
which are homotopy invariant,
are stable under the tensor product with the $C^*$-algebra of compact
operators on a separable Hilbert space,
have a long exact sequence for a short exact sequence,
satisfy the Bott periodicity
and have a Milnor exact sequence for a countable projective limit.
See \cite{MR1050490} for details.

In this appendix, we state the Milnor exact sequence by Phillips
and give a proof for reader's convenience.
He stated the following (in fact an equivariant version of the following)
in {\cite[Theorem 5.8 (5)]{MR1050490}}.



\begin{proposition}\label{Phillips}
Let $\{\pi_k:A_{k+1}\to A_k\}_{k\in\N}$ be 
a projective system of $\sigma$-$C^\ast$-algebras. 
Then we have the following functorial exact sequence for each $p\in \Z$.
\[
0\to \limone RK_{p+1}(A_k) 
\to RK_p(\varprojlim A_k)
\to  \varprojlim RK_p(A_k) 
\to  0.
\]
\end{proposition}
Phillips gives a proof 
under the condition that every $\pi_k$ is surjective 
{\cite[Theorem 3.2]{MR1050490}}. 
In order to prove, we refer to it and to \cite{MR1388297}.
\begin{proof}
We define 
\begin{align*}
T&:=\{(F_k)\in \prod_{k\in \N}C([k-1,k], A_k)
\left| F_k(k)=\pi_k(F_{k+1}(k)) \text{ for any }k\in \N\right.\},\\
B_{k+1}&:=\{(F_k,a_{k+1})\in C([k-1,k],A_k)\oplus A_{k+1}\left|
 F_k(k)=\pi_k(a_{k+1})\right.\},\\
g_1&:T\ni (F_k) \mapsto (F_{2m-1}, F_{2m}(2m-1)) \in \prod_{m\in \N}B_{2m},\\
g_2&:T\ni (F_k) \mapsto (F_1(0), (F_{2m}, F_{2m+1}(2m))) \in
 A_1\oplus\prod_{m\in\N}B_{2m+1},\\
f_1&:\prod_{m\in \N}B_{2m}\ni (F_1,a_2,F_3,a_4,\ldots)
\mapsto (F_1(0),a_2,F_3(2),a_4,\ldots)\in \prod_{k\in\N} A_k,\\
f_2&:A_1\oplus\prod_{m\in\N}B_{2m+1}\ni (a_1, F_2, a_3, F_4,\ldots)
\mapsto (a_1,F_2(1),a_3,F_3(2),\ldots) \in \prod_{k\in\N} A_k,\\
\iota&: \varprojlim A_k\ni (a_k)\mapsto ([k-1,k]\ni t\mapsto a_k)\in T,\\
\pi&: T\ni (F_k)\mapsto (F_k(k))\in \varprojlim A_k.
\end{align*}
We have a pullback diagram
\[
\xymatrix{
&T \ar[r]^{g_1} \ar[d]^{g_2}&\prod_{m\in \N}B_{2m} \ar[d]^{f_1} &\\
&A_1\oplus\prod_{m\in \N}B_{2m+1} \ar[r]^{f_2}&\prod_{k\in\N} A_k.& 
}\]
Hence we have a Mayer-Vietoris sequence. 
Since $\pi\circ \iota=id$ and also $\iota\circ \pi$ and $id$ are homotopic, 
$\iota$ gives a homotopy equivalence between the above
 pullback diagram and the following commutative diagram 
\[
\xymatrix{
&\varprojlim A_k \ar[r] \ar[d]&\prod_{m\in \N}A_{2m} \ar[d]&\\
&\prod_{m\in \N}A_{2m-1} \ar[r]&\prod_{k\in\N} A_k.& 
}\]
Now we have the desired functorial Milnor exact sequence. 
\end{proof}

\bibliographystyle{amsplain}
\bibliography{/Users/tomo/Library/tex/math}

\providecommand{\bysame}{\leavevmode\hbox to3em{\hrulefill}\thinspace}
\providecommand{\MR}{\relax\ifhmode\unskip\space\fi MR }
\providecommand{\MRhref}[2]{%
  \href{http://www.ams.org/mathscinet-getitem?mr=#1}{#2}
}
\providecommand{\href}[2]{#2}
\begin{thebibliography}{10}

\bibitem{MR1656031}
Bruce Blackadar, \emph{{$K$}-theory for operator algebras}, second ed.,
  Mathematical Sciences Research Institute Publications, vol.~5, Cambridge
  University Press, Cambridge, 1998. \MR{1656031 (99g:46104)}

\bibitem{MR2922380}
B.~H. Bowditch, \emph{Relatively hyperbolic groups}, Internat. J. Algebra
  Comput. \textbf{22} (2012), no.~3, 1250016, 66. \MR{2922380}

\bibitem{MR2364071}
Marius Dadarlat and Erik Guentner, \emph{Uniform embeddability of relatively
  hyperbolic groups}, J. Reine Angew. Math. \textbf{612} (2007), 1--15.
  \MR{2364071 (2008h:20064)}

\bibitem{Dahmani-2003}
Fran{\c{c}}ois Dahmani, \emph{Combination of convergence groups}, Geom. Topol.
  \textbf{7} (2003), 933--963 (electronic). \MR{2026551 (2005g:20063)}

\bibitem{MR1607744}
A.~N. Dranishnikov, J.~Keesling, and V.~V. Uspenskij, \emph{On the {H}igson
  corona of uniformly contractible spaces}, Topology \textbf{37} (1998), no.~4,
  791--803. \MR{MR1607744 (99k:57049)}

\bibitem{MR2225040}
Heath Emerson and Ralf Meyer, \emph{Dualizing the coarse assembly map}, J.
  Inst. Math. Jussieu \textbf{5} (2006), no.~2, 161--186. \MR{2225040
  (2007f:19007)}

\bibitem{EM-descent-Principle}
\bysame, \emph{A descent principle for the {D}irac--dual-{D}irac method},
  Topology \textbf{46} (2007), no.~2, 185--209. \MR{2313071 (2008f:57038)}

\bibitem{Far98}
B.~Farb, \emph{Relatively hyperbolic groups}, Geom. Funct. Anal. \textbf{8}
  (1998), no.~5, 810--840. \MR{1650094 (99j:20043)}

\bibitem{relhypgrp}
Tomohiro Fukaya and Shin-ichi Oguni, \emph{The coarse {B}aum-{C}onnes
  conjecture for relatively hyperbolic groups}, J. Topol. Anal. \textbf{4}
  (2012), no.~1, 99--113. \MR{2914875}

\bibitem{Busemann_cBC}
\bysame, \emph{The coarse {B}aum-{C}onnes conjecture for {B}usemann
  non-positively curved spaces}, arXiv:1304.3224 (2013).

\bibitem{product-cBC}
\bysame, \emph{Coronae of product spaces and the coarse baum-connes
  conjecture}, arXiv:1404.2770 (2014).

\bibitem{MR1086648}
{\'E}.~Ghys and P.~de~la Harpe (eds.), \emph{Sur les groupes hyperboliques
  d'apr\`es {M}ikhael {G}romov}, Progress in Mathematics, vol.~83, Birkh\"auser
  Boston Inc., Boston, MA, 1990, Papers from the Swiss Seminar on Hyperbolic
  Groups held in Bern, 1988. \MR{MR1086648 (92f:53050)}

\bibitem{MR919829}
M.~Gromov, \emph{Hyperbolic groups}, Essays in group theory, Math. Sci. Res.
  Inst. Publ., vol.~8, Springer, New York, 1987, pp.~75--263. \MR{919829
  (89e:20070)}

\bibitem{asym_invs}
\bysame, \emph{Asymptotic invariants of infinite groups}, Geometric group
  theory, {V}ol.\ 2 ({S}ussex, 1991), London Math. Soc. Lecture Note Ser., vol.
  182, Cambridge Univ. Press, Cambridge, 1993, pp.~1--295. \MR{1253544
  (95m:20041)}

\bibitem{MR2448064}
Daniel Groves and Jason~Fox Manning, \emph{Dehn filling in relatively
  hyperbolic groups}, Israel J. Math. \textbf{168} (2008), 317--429.
  \MR{2448064 (2009h:57030)}

\bibitem{MR1451755}
Nigel Higson, Erik~Kj{\ae}r Pedersen, and John Roe, \emph{{$C^\ast$}-algebras
  and controlled topology}, $K$-Theory \textbf{11} (1997), no.~3, 209--239.
  \MR{1451755 (98g:19009)}

\bibitem{MR1388312}
Nigel Higson and John Roe, \emph{On the coarse {B}aum-{C}onnes conjecture},
  Novikov conjectures, index theorems and rigidity, {V}ol.\ 2 ({O}berwolfach,
  1993), London Math. Soc. Lecture Note Ser., vol. 227, Cambridge Univ. Press,
  Cambridge, 1995, pp.~227--254. \MR{1388312 (97f:58127)}

\bibitem{MR1817560}
\bysame, \emph{Analytic {$K$}-homology}, Oxford Mathematical Monographs, Oxford
  University Press, Oxford, 2000, Oxford Science Publications. \MR{1817560
  (2002c:58036)}

\bibitem{MR1219916}
Nigel Higson, John Roe, and Guoliang Yu, \emph{A coarse {M}ayer-{V}ietoris
  principle}, Math. Proc. Cambridge Philos. Soc. \textbf{114} (1993), no.~1,
  85--97. \MR{MR1219916 (95c:19006)}

\bibitem{MR2379803}
Lizhen Ji, \emph{Integral {N}ovikov conjectures and arithmetic groups
  containing torsion elements}, Comm. Anal. Geom. \textbf{15} (2007), no.~3,
  509--533. \MR{2379803 (2009b:22010)}

\bibitem{Kap-Leeb-3mdf-gr}
M.~Kapovich and B.~Leeb, \emph{{$3$}-manifold groups and nonpositive
  curvature}, Geom. Funct. Anal. \textbf{8} (1998), no.~5, 841--852.
  \MR{1650098 (2000a:57040)}

\bibitem{MR1388297}
John Milnor, \emph{On the {S}teenrod homology theory}, Novikov conjectures,
  index theorems and rigidity, {V}ol.\ 1 ({O}berwolfach, 1993), London Math.
  Soc. Lecture Note Ser., vol. 226, Cambridge Univ. Press, Cambridge, 1995,
  pp.~79--96. \MR{1388297 (98d:55005)}

\bibitem{MR1050490}
N.~Christopher Phillips, \emph{Representable {$K$}-theory for
  {$\sigma$}-{$C^*$}-algebras}, $K$-Theory \textbf{3} (1989), no.~5, 441--478.
  \MR{1050490 (91k:46082)}

\bibitem{MR0507234}
M.~S. Raghunathan, \emph{Discrete subgroups of {L}ie groups}, Springer-Verlag,
  New York, 1972, Ergebnisse der Mathematik und ihrer Grenzgebiete, Band 68.
  \MR{0507234 (58 \#22394a)}

\bibitem{ROEHYPERBOLIC}
John Roe, \emph{Hyperbolic metric spaces and the exotic cohomology {N}ovikov
  conjecture}, $K$-Theory \textbf{4} (1990/91), no.~6, 501--512. \MR{1123175
  (93e:58180a)}

\bibitem{MR1147350}
\bysame, \emph{Coarse cohomology and index theory on complete {R}iemannian
  manifolds}, Mem. Amer. Math. Soc. \textbf{104} (1993), no.~497, x+90.
  \MR{MR1147350 (94a:58193)}

\bibitem{MR2007488}
\bysame, \emph{Lectures on coarse geometry}, University Lecture Series,
  vol.~31, American Mathematical Society, Providence, RI, 2003. \MR{MR2007488
  (2004g:53050)}

\bibitem{MR666554}
Edwin~H. Spanier, \emph{Algebraic topology}, Springer-Verlag, New York, 1981,
  Corrected reprint. \MR{666554 (83i:55001)}

\bibitem{WillettThesis}
Rufus Willett, \emph{Band-dominated operators and the stable higson corona},
  PhD thesis, Penn State (2009).

\bibitem{MR1344138}
Guoliang Yu, \emph{Coarse {B}aum-{C}onnes conjecture}, $K$-Theory \textbf{9}
  (1995), no.~3, 199--221. \MR{1344138 (96k:58214)}

\bibitem{MR1728880}
\bysame, \emph{The coarse {B}aum-{C}onnes conjecture for spaces which admit a
  uniform embedding into {H}ilbert space}, Invent. Math. \textbf{139} (2000),
  no.~1, 201--240. \MR{1728880 (2000j:19005)}

\end{thebibliography}

\bigskip
\address{ Tomohiro Fukaya \endgraf
Mathematical institute, Tohoku University, Sendai 980-8578, Japan}

\textit{E-mail address}: \texttt{tomo@math.tohoku.ac.jp}

\address{ Shin-ichi Oguni\endgraf
Department of Mathematics, Faculty of Science,
Ehime University,
2-5 Bunkyo-cho,
Matsuyama,
Ehime,
790-8577 Japan
}

\textit{E-mail address}: \texttt{oguni@math.sci.ehime-u.ac.jp}

\end{document}